%% file: AbsSIOV1.tex
\documentclass[11pt]{article}
\usepackage[a4paper,margin=1in]{geometry}
\usepackage[T1]{fontenc}
\usepackage[utf8]{inputenc}
\usepackage{amsmath,amssymb,amsthm,mathtools}

\usepackage[colorlinks]{hyperref}

\newtheorem{theorem}{Theorem}[section]
\newtheorem{lemma}[theorem]{Lemma}
\newtheorem{proposition}[theorem]{Proposition}
\newtheorem{corollary}[theorem]{Corollary}
\theoremstyle{definition}
\newtheorem{definition}[theorem]{Definition}
\newtheorem{remark}[theorem]{Remark}

\title{Coarea Reduction, Sparse Transfer, and Geometric Recomposition for Synchronized Singular Forms}
\author{Vicente Vergara\footnote{Department of Mathematics, Faculty of Physical and Mathematical Sciences, University of Concepci\'on, Concepci\'on, Chile. \texttt{vvergaraa@udec.cl}}}
\date{}

\begin{document}
	\maketitle

\begin{abstract}
	We study truncated bilinear forms associated with synchronized kernels
	\[
	K(x,y)=k(\phi(x),\psi(y)),
	\]
	where the singularity is governed by a one-dimensional kernel $k$, while the geometry is encoded by the phases $\phi$ and $\psi$. The central result of the paper is an architecture of exact reduction, analytic transfer, and geometric recomposition for this class of forms.

	First, we obtain an exact reduction at the level of pushforward measures and weighted pushforward measures in the level variable. Under absolute-continuity hypotheses, this reduction admits an effective realization in the Lebesgue layer, where control of the pushforward densities yields an abstract operator criterion for feeding estimates obtained in the reduced model back into the original problem.

	As a first complete realization of this scheme, we transfer to the synchronized setting a one-dimensional sparse domination principle for singular truncations with Dini-smooth kernels. The final geometric recomposition then separates two regimes: a uniform regime, where global consequences follow from quantitative control of the pushforward densities, and a critical regime, where degeneration of the phases near critical values forces a localized output weighted by pullbacks.
\end{abstract}

\input{00_intro_abs-sio}
\input{10_preli_abs-sio}
\input{20_id_abs-sio}
\input{25_sparse_abs-sio}
\input{26_pv_maximal_capsule}
\input{27_pushforward_lp_bridge}
\input{30_geom_abs-sio}
\input{35_geom_abs-sio}
\input{40_main_thm_abs-sio}
\input{50_examples_abs-sio}

\end{document}

%% file: 00_intro_abs-sio.tex
\section{Introduction}

We study truncated singular bilinear forms on open sets $\Omega_x,\Omega_y\subset\mathbb{R}^n$ of the form
\begin{equation}\label{eq:formabilineal:0}
\langle T_\varepsilon f,g\rangle
=
\iint_{\Omega_x\times\Omega_y}
\mathbf{1}_{\{|\phi(x)-\psi(y)|>\varepsilon\}}
\,K(x,y)\,f(y)\,g(x)\,\mathrm{d}y\,\mathrm{d}x,
\end{equation}
in the case where the kernel is \emph{synchronized} by two real phases $\phi$ and $\psi$, namely
\[
K(x,y)=k(\phi(x),\psi(y)),
\]
for a one-dimensional singular kernel $k$ defined off the diagonal. In this regime, the singularity is governed by the level variable $|\phi(x)-\psi(y)|$, while the geometry of the problem is concentrated on the fibers and level sets of $\phi$ and $\psi$.

The goal of the manuscript is to isolate this structure. We prove that the synchronized truncated form first admits an exact reduction to a one-dimensional singular form in the level variables and that, after applying a sparse input in that reduced layer, the output can be geometrically recomposed in the original space. The main architecture is therefore
\[
\text{exact pushforward reduction}
\quad\Longrightarrow\quad
\text{sparse transfer and geometric recomposition}.
\]
This organization yields two main results: Theorem~\ref{thm:exact-pushforward-reduction}, which establishes the exact reduction, and Theorem~\ref{thm:main-recomposition}, which recomposes the transferred output and separates the uniform and critical regimes.

One-dimensional sparse domination is not used here as an isolated result, but rather as an analytic input within an architecture of reduction and recomposition. The structural point is that, for synchronized kernels, the multidimensional singularity separates exactly from the geometry of the fibers: first one reduces to the pushforward level, then one applies one-dimensional singular analysis, and finally one recomposes the output in terms of the original phases. This recomposition gives rise to two qualitatively different behaviors: a \textit{uniform regime}, with global functional closure, and a \textit{critical regime}, whose natural output is localized and weighted.

This perspective is related to the line initiated by Maz'ya, who showed that certain classes of multidimensional integral equations admit an equivalent reduction to equations for functions of fewer variables through a factorization based on the coarea formula \cite{Mazya2007CoareaIE}. The affinity with that point of view is structural: here as well, reduction to the level variable is the organizing core of the analysis.

The difference lies in the point at which the geometry is introduced. In Maz'ya's strong formulation for two phases, the recomposition already incorporates an explicit geometric restriction between the two families of levels, namely the non-parallelism of the normals on the intersections
\[
\{\phi=\sigma\}\cap\{\psi=\tau\},
\]
or equivalently
\[
\nabla\phi(P)\wedge\nabla\psi(P)\neq 0,
\qquad\text{that is,}\qquad
\sin\omega(P)\neq 0,
\]
where $\omega(P)$ denotes the angle between $\nabla\phi(P)$ and $\nabla\psi(P)$. In that framework, the reduced kernel explicitly contains the factor
\[
\frac{1}{|\nabla\psi(P)|\,|\nabla\phi(P)|\,|\sin\omega(P)|},
\]
so that the transversality between the two phases is built into the strong reduction from the outset.

Our approach separates three operations that should be kept distinct:
\begin{itemize}
	\item the exact reduction of the bilinear form at the level of pushforward measures;
	\item the analytic transfer inside the one-dimensional singular model;
	\item the final geometric recomposition in terms of fibers, critical values, and pushforward densities.
\end{itemize}
The concrete technology used in this first realization is sparse domination for Dini-smooth kernels in one dimension \cite{BallestaCondeAlonso2025DiniSparse}.

The examples in Section~\ref{sec:examples-regimes} show that this separation reflects genuinely distinct geometric and analytic mechanisms. In particular, they distinguish the role of critical values, the geometric uniformity of the fibers, and the possibility of global functional closure.

\subsection*{Main results}

The results of the manuscript are organized around two theorems.

\paragraph{Theorem~\ref{thm:exact-pushforward-reduction}: exact pushforward reduction.}
The first piece of the framework consists in rewriting the synchronized bilinear form
\eqref{eq:formabilineal:0} as a one-dimensional singular form built over the pushforward measures associated with $\phi$ and $\psi$. If $\nu_\phi$ and $\nu_\psi$ denote the corresponding pushforward measures, we introduce
\begin{equation}\label{eq:intro-stage1-Lambda-nu}
\Lambda_\varepsilon^\nu(F,G)
:=
\int_{\mathbb{R}}\int_{\mathbb{R}}
\mathbf{1}_{\{|s-t|>\varepsilon\}}\,k(s,t)\,F(t)\,G(s)
\,\mathrm{d}\nu_\psi(t)\,\mathrm{d}\nu_\phi(s).
\end{equation}
At this intrinsic level, the original data are transferred to the level variable through the relative densities $Q_{\psi,f}$ and $Q_{\phi,g}$ associated with the weighted pushforward measures. Theorem~\ref{thm:exact-pushforward-reduction} establishes the exact identity
\begin{equation}\label{eq:intro-stage1-reduction}
	\langle T_\varepsilon f,g\rangle
	=
	\Lambda_\varepsilon^\nu(Q_{\psi,f},Q_{\phi,g}).
\end{equation}

This identity provides the structural link between the original geometric form and the one-dimensional singular problem. In particular, the truncated form is not merely compared with a one-dimensional model; it is identified exactly with it at the level of pushforward measures. This is the content of Theorem~\ref{thm:exact-pushforward-reduction}: the singular analysis is concentrated in the level variable, while the original geometry is reserved for the later stage of transfer and recomposition.

\paragraph{Lebesgue realization and the pushforward operator bridge.}
When the pushforward measures associated with $\phi$ and $\psi$ are absolutely continuous, Corollary~\ref{cor:reduction-lebesgue} rewrites the previous identity in the form
\begin{equation}\label{eq:intro-stage2-lebesgue}
\langle T_\varepsilon f,g\rangle
=
\int_{\mathbb{R}}\int_{\mathbb{R}}
\mathbf{1}_{\{|s-t|>\varepsilon\}}\,k(s,t)
\,w_{\psi,f}(t)\,w_{\phi,g}(s)
\,\mathrm{d}t\,\mathrm{d}s.
\end{equation}
Here $w_{\psi,f}$ and $w_{\phi,g}$ represent the pushforward densities associated with the data $f$ and $g$ in the level variable. In this effective layer, Section~\ref{sec:pushforward-lp-bridge} reformulates the reduction in terms of the pushforward operator and establishes an abstract \(L^p\)-boundedness criterion under geometric control of the pushforward densities.

\paragraph{One-dimensional analytic module and sparse transfer.}
The next piece develops the one-dimensional analytic module for the hard truncation and formulates its sparse transfer to the truncated geometric operator. Section~\ref{sec:sparse-1d} organizes this step through the smoothed family
\[
T^{1D}_{\varepsilon,\mathrm{sm}}F(s)
=
\int_{\mathbb{R}}k(s,t)\chi(|s-t|/\varepsilon)F(t)\,\mathrm{d}t.
\]
The kernel of this family uniformly satisfies the Calder\'on--Zygmund operator hypotheses with a modulus of continuity $\omega$ satisfying the Dini condition
\[
\int_0^1 \omega(u)\,\frac{\mathrm{d}u}{u}<\infty.
\]
The one-dimensional sparse result of \cite{BallestaCondeAlonso2025DiniSparse} is applied to this family, and the hard--smooth error is then controlled separately in order to return to the hard truncation. The operative output of this piece is the sparse transfer to the truncated geometric operator, formulated in Corollary~\ref{cor:sparse-transfer-Teps}.

\paragraph{Theorem~\ref{thm:main-recomposition}: sparse transfer and geometric recomposition.}
The final piece recomposes the transferred sparse output in terms of the original phases. Theorem~\ref{thm:main-recomposition} starts from the sparse output in the level variable and obtains functional consequences for the truncated geometric family. In its abstract form, if
\[
A_\psi f\in L^r(\mathbb{R}),
\qquad
A_\phi g\in L^{r'}(\mathbb{R}),
\]
then
\[
|\langle T_\varepsilon f,g\rangle|
\le
C_r\,\|A_\psi f\|_{L^r(\mathbb{R})}\,\|A_\phi g\|_{L^{r'}(\mathbb{R})}.
\]
The geometric part of the theorem identifies two ways of closing these norms. In the uniform regime, local control on level intervals is globalized by a finite covering and produces a global unweighted bound. In the critical regime, the pushforward density may degenerate near critical values, and the natural output becomes localized and weighted by pullback.

The coarea representation leads to expressions of the form
\[
w_\theta(t)
=
\int_{\{\theta=t\}}\frac{1}{|\nabla \theta(x)|}\,\mathrm{d}\mathcal{H}^{n-1}(x),
\qquad
w_{\theta,h}(t)
=
\int_{\{\theta=t\}} h(x)\,\frac{1}{|\nabla \theta(x)|}\,\mathrm{d}\mathcal{H}^{n-1}(x),
\]
where $\theta$ denotes either of the phases $\phi$ or $\psi$. These formulas show that recomposition depends on the behavior of the fibers and on the quantitative law by which the phase separates adjacent levels. In analytic classes, such profiles arise naturally through inequalities of \L{}ojasiewicz type \cite{Chill2003,NguyenPham2022}. When this control deteriorates near critical values, the denominator may amplify the pushforward density and force a localized output.

\subsection*{Organization of the manuscript}

Section~\ref{sec:measurable-core} fixes the notation and the preliminary facts used throughout the manuscript. There we introduce the \textbf{(Hk)} package for the one-dimensional kernel, the formulation in terms of pushforward measures, the coarea and disintegration machinery, and the fiber notation that later enters the Lebesgue formulation and the recomposition.

Section~\ref{sec:reduction-1d} contains the first main result, Theorem~\ref{thm:exact-pushforward-reduction}. It first formulates the reduction at the robust level of pushforward measures and then, under additional absolute-continuity hypotheses, obtains the Lebesgue formulation of Corollary~\ref{cor:reduction-lebesgue}.

Section~\ref{sec:sparse-1d} contains the one-dimensional sparse module needed for the hard truncation. It verifies that the reduced smoothed family falls uniformly in the relevant CZO-Dini class, applies the result of \cite{BallestaCondeAlonso2025DiniSparse} to that family, and treats the hard--smooth error in order to recover the hard truncation and transfer the output to the truncated geometric operator through \eqref{eq:theorem-A-lebesgue}.

Section~\ref{sec:pv-maximal-capsule} collects complementary observations on the robustness of truncations. There we separate cutoff independence, the maximal comparison between hard and smooth truncations, and the status of future principal-value assertions.

Section~\ref{sec:pushforward-lp-bridge} isolates the pushforward operator criterion in the Lebesgue layer and converts the effective hypothesis of the sparse transfer into a structural consequence under explicit geometric hypotheses.

Section~\ref{sec:fiber-analytic-objective} begins the geometric recomposition. There the transferred output is rewritten in terms of fiber operators, a first local output regime on level intervals is obtained, and the geometric principle opening the uniform regime is formulated through quantitative non-degeneracy profiles, more flexible than strict uniform submersion.

Section~\ref{sec:fiber-critical-regime} analyzes the critical scenario. There the geometric obstruction to global uniform closure is identified, the partial integrability of the critical profile is recorded, and the corresponding localized weighted output is obtained.

Section~\ref{sec:main-results} contains the second main result, Theorem~\ref{thm:main-recomposition}. This section recomposes the preceding modules into a structural principle of sparse transfer and geometric recomposition, and makes explicit the final dichotomy between a global output in the uniform regime and a localized output in the critical regime.

Section~\ref{sec:examples-regimes} closes the manuscript with examples that distinguish the geometric and analytic mechanisms responsible for the two behaviors. In particular, it separates the role of the critical value, the blow-up of the pushforward density, the geometric uniformity of the fibers, and the possibility of global functional closure.

%% file: 10_preli_abs-sio.tex
\section{Preliminaries}\label{sec:measurable-core}

In this section we fix the notation and the preliminary facts used throughout the manuscript. We collect, on the one hand, the one-dimensional kernel $k$ and its truncations and, on the other hand, the pushforward measures associated with a phase function $\theta$, their densities when they exist, and the corresponding fiber operators. Longer proofs are postponed to later sections.

Associated with the phase $\theta$, we introduce the pushforward measure
\begin{equation}\label{eq:def-nu-theta}
	\nu_\theta:=\theta_\#(\mathrm{d}x\llcorner\Omega),
\end{equation}
that is,
\begin{equation}\label{eq:def-nu-theta-borel}
	\nu_\theta(E)
	=
	\bigl|\{x\in\Omega:\theta(x)\in E\}\bigr|
\end{equation}
for every Borel set $E\subset\mathbb{R}$.

More generally, if $\rho:\Omega\to[0,\infty)$ is measurable with $\rho\in L^1(\Omega)$, we define the weighted pushforward measure
\begin{equation}\label{eq:def-nu-theta-rho}
	\nu_{\theta,\rho}:=\theta_\#(\rho\,\mathrm{d}x\llcorner\Omega),
\end{equation}
namely
\begin{equation}\label{eq:def-nu-theta-rho-borel}
	\nu_{\theta,\rho}(E)
	=
	\int_{\{x\in\Omega:\theta(x)\in E\}}\rho(x)\,\mathrm{d}x
\end{equation}
for every Borel set $E\subset\mathbb{R}$. In the applications we shall mainly use the case $\rho=|f|^r$, with
$f\in L^r(\Omega)$ and $1\le r<\infty$, so that $\rho\in L^1(\Omega)$ and the preceding definition is well posed without additional hypotheses. See, for instance, \cite{EvansGariepy} for a general reference on measure theory, Hausdorff measure, and area/coarea formulae in $\mathbb{R}^n$.

\subsection{One-dimensional kernel, truncations, and analytic package}\label{subsec:kernel-1d}

We fix a measurable kernel
\[
k:\mathbb{R}^2\setminus\{s=t\}\to\mathbb{C}.
\]
Formally, we consider the one-dimensional singular operator
\begin{equation}\label{eq:T1d-formal}
	(T^{1D}f)(s):=\int_{\mathbb{R}} k(s,t)\,f(t)\,\mathrm{d}t,
\end{equation}
interpreted through truncations. Initially, the operators are defined on $L^\infty_c(\mathbb{R})$ and, when appropriate, extended by density to the relevant spaces $L^p(\mathbb{R})$.

We fix a cutoff function $\chi\in C^\infty([0,\infty))$ such that
$0\le \chi\le 1$, $\chi(r)=0$ for $0\le r\le 1$, and $\chi(r)=1$ for $r\ge 2$.

\begin{definition}[Hard and smooth truncations]\label{def:trunc-hard-soft}
	For $\varepsilon>0$ we define
	\begin{align}
		(T^{1D}_\varepsilon f)(s)
		&:=\int_{|s-t|>\varepsilon} k(s,t)\,f(t)\,\mathrm{d}t,\label{eq:T-hard}\\
		(T^{1D}_{\varepsilon,\mathrm{sm}} f)(s)
		&:=\int_{\mathbb{R}} k(s,t)\,\chi\!\left(\frac{|s-t|}{\varepsilon}\right)f(t)\,\mathrm{d}t.\label{eq:T-soft}
	\end{align}
	For the difference between these truncations, we introduce
	\begin{equation}\label{eq:R1d-def}
		R^{1D}_\varepsilon:=T^{1D}_\varepsilon-T^{1D}_{\varepsilon,\mathrm{sm}},
	\end{equation}
	which compares the hard and smooth truncations. We also define the hard maximal truncation by
	\begin{equation}\label{eq:T-star}
		(T^{1D,\ast}f)(s):=\sup_{\varepsilon>0}|(T^{1D}_\varepsilon f)(s)|.
	\end{equation}
\end{definition}

\subsection*{Analytic package for $k$.}
\begin{itemize}
	\item[\textbf{(Hk1)}\label{hyp:k-size}] \textbf{Size of the one-dimensional kernel.} There exists $C_{k1}>0$ such that
	\begin{equation}\label{eq:k-size}
		|k(s,t)|\le \frac{C_{k1}}{|s-t|}
		\qquad\text{for all }s\neq t.
	\end{equation}

	\item[\textbf{(Hk2)}] \textbf{Dini regularity off the diagonal.}
	There exists an increasing modulus of continuity $\omega:[0,1]\to[0,\infty)$, with
	$\omega(0)=0$, such that
	\begin{equation}\label{eq:dini-modulus}
		\int_0^1 \omega(u)\,\frac{\mathrm{d}u}{u}<\infty,
	\end{equation}
	and, moreover, for all $s,s',t\in\mathbb{R}$ with
	$|s-s'|\le \frac12 |s-t|$, one has
	\begin{equation}\label{eq:kernel-smooth-1d-s}
		|k(s,t)-k(s',t)|
		\le
		\omega\!\left(\frac{|s-s'|}{|s-t|}\right)\frac{1}{|s-t|},
	\end{equation}
	while, for all $s,t,t'\in\mathbb{R}$ with
	$|t-t'|\le \frac12 |s-t|$, one has
	\begin{equation}\label{eq:kernel-smooth-1d-t}
		|k(s,t)-k(s,t')|
		\le
		\omega\!\left(\frac{|t-t'|}{|s-t|}\right)\frac{1}{|s-t|}.
	\end{equation}

	\item[\textbf{(Hk3)}\label{hyp:k-L2}] \textbf{Uniform $L^2$ input for smooth truncations.} There exists $C_{k3}>0$ such that
	\begin{equation}\label{eq:k-L2}
		\sup_{\varepsilon>0}\|T^{1D}_{\varepsilon,\mathrm{sm}}\|_{L^2(\mathbb{R})\to L^2(\mathbb{R})}\le C_{k3}.
	\end{equation}
\end{itemize}

We shall write \textbf{(Hk)} for the conjunction of \textbf{(Hk1)},
\textbf{(Hk2)}, and \textbf{(Hk3)}.

\begin{remark}[Origin and dependencies of the \textbf{(Hk)} package]\label{rem:Hk-origin}
	The conditions \textbf{(Hk1)}--\textbf{(Hk3)} lie in the standard framework of Calder\'on--Zygmund-type singular operators with Dini regularity and in the literature on maximal truncations and sparse domination; see, for instance,
	\cite{HytonenRoncalTapiola2015Rough,AndersonHu2020MaxTrunc,CondeAlonsoParcet2016Nondoubling}.
	In the present work, this package is adopted as the initial analytic hypothesis for the later one-dimensional module. Consequently, every constant coming from that module will depend only on $C_{k1}$, on $C_{k3}$, on the Dini functional
	\[
	\int_0^1 \omega(u)\,\frac{\mathrm{d}u}{u},
	\]
	and on the fixed choice of the cutoff $\chi$.
\end{remark}

\subsection{Geometric packages}\label{subsec:hyp-packages}

\noindent\textbf{Geometric packages.}
For $t\in\mathbb R$, we write
\[
\Sigma_t:=\theta^{-1}(t).
\]
\begin{itemize}
	\item[\textbf{(H1)}\label{hyp:H1}] \emph{Quantitative submersion in a level tube.} There exist an open interval $I_0\subset\mathbb{R}$ and an open neighborhood $U\subset\mathbb{R}^n$ with
	\[
	\{x\in\overline\Omega:\theta(x)\in I_0\}\subset U,
	\]
	such that $\theta$ is at least $C^{1,1}$ on $U$ and
	\begin{equation}\label{eq:H1-nondeg}
		|\nabla\theta(x)|\ge c_0>0
		\qquad\text{for all }x\in U.
	\end{equation}
	This package excludes critical values in the tube and allows one to discuss pointwise properties of $w_\theta$ and $w_{\theta,\rho}$ on subintervals of $I_0$, without by itself asserting any additional uniform regularity.

	\item[\textbf{(H2)}\label{hyp:H2}] \emph{Quantitative trivialization of the tube.} There exist an interval $I_0=(a,b)$, a compact $(n-1)$-dimensional manifold $Y$, and a map $F:I_0\times Y\to U\cap\Omega$ of class at least $C^1$ such that: (i) $\theta(F(t,y))=t$ for all $(t,y)\in I_0\times Y$; (ii) for each $t\in I_0$, the map $F_t:=F(t,\cdot)$ parametrizes $\Sigma_t\cap\Omega$; (iii) the tangential Jacobian $J(t,y):=J_{n-1}(D_yF(t,y))$ is uniformly bounded above and below; and (iv) the transverse factor $|\nabla\theta(F(t,y))|^{-1}$ is quantitatively controlled. Then
	\begin{equation}\label{eq:wtheta-rho-chart}
		w_{\theta,\rho}(t)
		=
		\int_Y \rho(F(t,y))\,\frac{J(t,y)}{|\nabla\theta(F(t,y))|}\,\mathrm{d}y,
		\qquad t\in I_0,
	\end{equation}
	which is the basic coordinate representation for obtaining regularity and pointwise bounds under additional hypotheses on $\rho$.

	\item[\textbf{(H3)}\label{hyp:H3}] \emph{Boundary contact and quantitative transversality.} In addition to \textbf{(H1)}, we assume that $\partial\Omega$ has the required geometric regularity in the relevant contact region and that there exists $c_\partial>0$ such that
	\begin{equation}\label{eq:H3-boundary-transversality}
		|\nabla_{\partial\Omega}\theta(x)|\ge c_\partial
		\qquad\text{for }\mathcal H^{n-1}\text{-a.e. }x\in U\cap\partial\Omega.
	\end{equation}
	This package excludes degenerate tangencies between the fibers $\Sigma_t$ and the boundary in the tube under consideration, and is the natural assumption for uniform control of the geometry of $\Sigma_t\cap\Omega$ when the boundary is involved.
\end{itemize}

\begin{remark}[Reading dependencies and scope]\label{rem:density-regularity-scope}
	In the measurable framework fixed at the beginning of this section, we only assert measurable disintegration, existence of pushforward measures, and integrated control of the fiber operators. Pointwise properties of $w_\theta$ and $w_{\theta,\rho}$ begin to be discussed under \textbf{(H1)} and admit a coordinate treatment under \textbf{(H2)}; when the boundary is involved, \textbf{(H3)} must be added. In particular, continuity, Lipschitz regularity, sup/inf bounds, and doubling properties require supplementary hypotheses and are not part of this preliminary framework.
\end{remark}

\subsection{Coarea, disintegration, and pushforward measures}\label{subsec:coarea-disintegration}

In this subsection we fix the notation for the weighted pushforward measures associated with \(\theta\). In addition to the positive measure $\nu_\theta$, we shall need to consider measures weighted by possibly signed or complex test functions.

\begin{equation}\label{eq:coarea}
	\int_\Omega g(x)\,|\nabla\theta(x)|\,\mathrm{d}x
	=
	\int_{\mathbb{R}}
	\left(
	\int_{\Sigma_t} g(x)\,\mathrm{d}\mathcal H^{n-1}(x)
	\right)\mathrm{d}t
\end{equation}
for every nonnegative measurable function $g$.

Applying \eqref{eq:coarea} to
\[
g(x)=\frac{h(x)\,\mathbf 1_{\{|\nabla\theta(x)|>0\}}}{|\nabla\theta(x)|},
\]
one obtains the operational form
\begin{equation}\label{eq:coarea-operational-form}
	\int_\Omega h(x)\,\mathbf 1_{\{|\nabla\theta(x)|>0\}}\,\mathrm{d}x
	=
	\int_{\mathbb{R}}
	\left(
	\int_{\Sigma_t}\frac{h(x)}{|\nabla\theta(x)|}\,\mathbf 1_{\{|\nabla\theta(x)|>0\}}
	\,\mathrm{d}\mathcal H^{n-1}(x)
	\right)\mathrm{d}t,
\end{equation}
whenever the right-hand side is meaningful.

We return to the pushforward measures $\nu_\theta$ and $\nu_{\theta,\rho}$ defined in \eqref{eq:def-nu-theta} and \eqref{eq:def-nu-theta-rho}. Then, for every bounded Borel-measurable function
$\zeta:\mathbb{R}\to\mathbb{C}$ one has
\begin{equation}\label{eq:pushforward-duality}
	\int_{\mathbb{R}} \zeta(t)\,\mathrm{d}\nu_\theta(t)
	=
	\int_\Omega \zeta(\theta(x))\,\mathrm{d}x,
\end{equation}
and, analogously,
\begin{equation}\label{eq:pushforward-identity-rho}
	\int_{\mathbb{R}} \zeta(t)\,\mathrm{d}\nu_{\theta,\rho}(t)
	=
	\int_\Omega \zeta(\theta(x))\,\rho(x)\,\mathrm{d}x.
\end{equation}

These identities are the basic formulation of disintegration along the fibers of $\theta$ and do not, by themselves, require $\nu_\theta$ or $\nu_{\theta,\rho}$ to be absolutely continuous with respect to Lebesgue measure.

Let \(\Omega\subset \mathbb{R}^n\) be measurable, let \(\theta:\Omega\to\mathbb{R}\) be measurable, and let \(h\in L^1(\Omega)\), possibly complex-valued. We define the weighted pushforward measure \(\nu_{\theta,h}\) on \(\mathbb{R}\) by
\begin{equation}\label{eq:def-nu-theta-h}
	\nu_{\theta,h}(E)
	:=
	\int_{\theta^{-1}(E)} h(x)\,\mathrm{d}x,
	\qquad E\subset\mathbb{R}\ \text{Borel}.
\end{equation}
Equivalently, for every bounded Borel function \(F:\mathbb{R}\to\mathbb{C}\),
\begin{equation}\label{eq:def-nu-theta-h-borel}
	\int_{\mathbb{R}} F(t)\,\mathrm{d}\nu_{\theta,h}(t)
	=
	\int_{\Omega} F(\theta(x))\,h(x)\,\mathrm{d}x.
\end{equation}

\begin{lemma}\label{lem:nu-theta-h-ac}
	Let \(h\in L^1(\Omega)\). Then \(\nu_{\theta,h}\) is a finite complex measure on \(\mathbb{R}\) and satisfies
	\begin{equation}\label{eq:nu-theta-h-ac}
		|\nu_{\theta,h}|(E)
		\le
		\int_{\theta^{-1}(E)} |h(x)|\,\mathrm{d}x
	\end{equation}
	for every Borel set \(E\subset\mathbb{R}\). In particular, if \(\nu_\theta(E)=0\), then \(\nu_{\theta,h}(E)=0\), that is,
	\[
	\nu_{\theta,h}\ll \nu_\theta.
	\]
\end{lemma}

\begin{proof}
	Finiteness follows immediately from \(h\in L^1(\Omega)\), since
	\[
	|\nu_{\theta,h}|(\mathbb{R})
	\le
	\int_\Omega |h(x)|\,\mathrm{d}x
	<
	\infty.
	\]
	The inequality \eqref{eq:nu-theta-h-ac} follows directly from the definition \eqref{eq:def-nu-theta-h}. If \(\nu_\theta(E)=0\), then \(|\theta^{-1}(E)|=0\), and therefore
	\[
	\nu_{\theta,h}(E)
	=
	\int_{\theta^{-1}(E)} h(x)\,\mathrm{d}x
	=
	0.
	\]
\end{proof}

\begin{definition}\label{def:relative-density-pushforward}
	Under the hypotheses of Lemma~\ref{lem:nu-theta-h-ac}, we define the relative pushforward density of \(h\) with respect to \(\theta\) as the Radon--Nikodym derivative
	\[
	Q_{\theta,h}
	:=
	\frac{\mathrm{d}\nu_{\theta,h}}{\mathrm{d}\nu_\theta}.
	\]
\end{definition}

In particular, for every bounded Borel function \(F:\mathbb{R}\to\mathbb{C}\),
\begin{equation}\label{eq:pushforward-relative-density}
	\int_{\mathbb{R}} F(t)\,\mathrm{d}\nu_{\theta,h}(t)
	=
	\int_{\mathbb{R}} F(t)\,Q_{\theta,h}(t)\,\mathrm{d}\nu_\theta(t).
\end{equation}

When in addition \(\nu_\theta\ll \mathrm{d}t\), we write
\[
\mathrm{d}\nu_\theta(t)=w_\theta(t)\,\mathrm{d}t.
\]
If also \(\nu_{\theta,h}\ll \mathrm{d}t\), we write
\[
\mathrm{d}\nu_{\theta,h}(t)=w_{\theta,h}(t)\,\mathrm{d}t.
\]
In that case,
\begin{equation}\label{eq:Q-theta-h-lebesgue}
	Q_{\theta,h}(t)
	=
	\frac{w_{\theta,h}(t)}{w_\theta(t)}
	\qquad
	\text{for \(\nu_\theta\)-a.e. }t\text{ such that }0<w_\theta(t)<\infty.
\end{equation}

In the same absolutely continuous setting, \eqref{eq:pushforward-duality} takes the form
\[
\int_{\mathbb{R}} \zeta(t)\,\mathrm{d}\nu_\theta(t)
=
\int_{\mathbb{R}} \zeta(t)\,w_\theta(t)\,\mathrm{d}t.
\]
Similarly, if
\[
\nu_{\theta,\rho}\ll \mathrm{d}t,
\]
we write
\[
\mathrm{d}\nu_{\theta,\rho}(t)=w_{\theta,\rho}(t)\,\mathrm{d}t.
\]

In the absolutely continuous regime, the coarea formula identifies these densities with fiber integrals: for almost every $t\in\mathbb{R}$,
\begin{equation}\label{eq:wtheta-coarea}
	w_\theta(t)
	=
	\int_{\Sigma_t}\frac{\mathbf 1_{\{|\nabla\theta|>0\}}(x)}{|\nabla\theta(x)|}\,
	\mathrm{d}\mathcal H^{n-1}(x),
\end{equation}
and, more generally,
\begin{equation}\label{eq:wtheta-rho-def}
	w_{\theta,\rho}(t)
	=
	\int_{\Sigma_t}\frac{\rho(x)\,\mathbf 1_{\{|\nabla\theta|>0\}}(x)}{|\nabla\theta(x)|}\,
	\mathrm{d}\mathcal H^{n-1}(x).
\end{equation}

\begin{remark}[Scope of the density formulation]
	The preceding representation by densities is not automatic for an arbitrary Lipschitz map.
	In particular, the pushforward measure $\nu_\theta$ may have a singular part with respect to Lebesgue measure, for instance if $\theta$ is constant on a set of positive measure.
	Thus, the language in terms of $w_\theta$ and $w_{\theta,\rho}$ is always understood either under an additional absolute-continuity hypothesis or in geometric regimes where that property has already been verified.
\end{remark}

These identities form the operational version of disintegration along the fibers of $\theta$. See, for instance, \cite{EvansGariepy,FedererGMT} for classical references on the coarea formula and its consequences in geometric measure theory.

\subsection{Critical values and fiber operators}\label{subsec:critical-values-fiber-operators}

We define the set of critical values of $\theta$ by
\begin{equation}\label{eq:critical-values}
	V_\theta:=\theta\bigl(\{x\in\Omega:\nabla\theta(x)=0\}\bigr).
\end{equation}
In this section, the role of $V_\theta$ is mainly organizational: it separates the minimal regime of measurable disintegration from the regime of quantitative submersion, in which pointwise properties of $w_\theta$ and $w_{\theta,\rho}$ may be discussed away from critical values. We do not impose here any additional hypotheses on the fine structure of $V_\theta$.

The structural object fixed above is the relative pushforward density
\(Q_{\theta,h}\). The present subsection introduces the two formulations that will be used systematically later: the normalized average over fibers and the effective weighted form \(w_\theta\,\widetilde M_\theta h\).

Motivated by \eqref{eq:wtheta-rho-def}, and working in the regime where
$\nu_{\theta,|f|^r}\ll \mathrm{d}t$, for $1\le r<\infty$ and
$f\in L^r(\Omega)$ we define the fiber operator
\begin{equation}\label{eq:Mtheta-def}
	M_\theta f(t)
	:=
	\left(
	\int_{\Sigma_t}\frac{|f(x)|^r\,\mathbf 1_{\{|\nabla\theta|>0\}}(x)}{|\nabla\theta(x)|}\,
	\mathrm{d}\mathcal H^{n-1}(x)
	\right)^{1/r}
	=
	w_{\theta,|f|^r}(t)^{1/r},
\end{equation}
for almost every $t$.

In the same regime, the pushforward identity implies
\begin{equation}\label{eq:Mtheta-Lp}
	\|M_\theta f\|_{L^r(\mathbb{R})}^r
	=
	\int_{\mathbb{R}} w_{\theta,|f|^r}(t)\,\mathrm{d}t
	=
	\int_\Omega |f(x)|^r\,\mathrm{d}x,
\end{equation}
and consequently
\begin{equation}\label{eq:Mtheta-Lp-root}
	\|M_\theta f\|_{L^r(\mathbb{R})}
	=
	\|f\|_{L^r(\Omega)}.
\end{equation}

When in addition $\nu_\theta\ll \mathrm{d}t$, we introduce the normalized average over the fiber
\begin{equation}\label{eq:normalized-average}
	(\widetilde M_\theta f)(t):=
	\begin{cases}
		\displaystyle
		\frac{1}{w_\theta(t)}
		\int_{\Sigma_t}\frac{f(x)}{|\nabla\theta(x)|}\,\mathrm{d}\mathcal H^{n-1}(x),
		& \text{if } 0<w_\theta(t)<\infty,\\[0.8em]
		0, & \text{if } w_\theta(t)=0 \text{ or } w_\theta(t)=\infty.
	\end{cases}
\end{equation}
In particular,
\begin{equation}\label{eq:wtheta-times-Mtilde}
	w_\theta(t)\,\widetilde M_\theta f(t)
	=
	\int_{\Sigma_t}\frac{f(x)}{|\nabla\theta(x)|}\,\mathrm{d}\mathcal H^{n-1}(x)
\end{equation}
for almost every $t$ with $0<w_\theta(t)<\infty$. If $1<r<\infty$ and $r'$ denotes the conjugate exponent of $r$, then H\"older's inequality on the fiber gives
\begin{equation}\label{eq:holder-fiber}
	\left|
	\int_{\Sigma_t}\frac{f(x)}{|\nabla\theta(x)|}\,\mathrm{d}\mathcal H^{n-1}(x)
	\right|
	\le
	w_\theta(t)^{1/r'}\,M_\theta f(t),
\end{equation}
and hence
\begin{equation}\label{eq:holder-fiber-nu}
	|\widetilde M_\theta f(t)|
	\le
	w_\theta(t)^{-1/r}\,M_\theta f(t)
\end{equation}
for almost every $t$ with $0<w_\theta(t)<\infty$.

\begin{remark}[Dictionary between relative density, normalized average, and effective form]
	\label{rem:Q-vs-M}
	Under the hypotheses in which \eqref{eq:Q-theta-h-lebesgue}
	and \eqref{eq:normalized-average} hold, the relative pushforward density
	\(Q_{\theta,h}\) coincides with the normalized average over fibers:
	\[
	Q_{\theta,h}(t)=\widetilde M_\theta h(t)
	\]
	for \(\nu_\theta\)-almost every \(t\) such that \(0<w_\theta(t)<\infty\).
	Consequently,
	\[
	w_{\theta,h}(t)=w_\theta(t)\,\widetilde M_\theta h(t)
	\]
	for almost every \(t\) in the same regime.
\end{remark}

\subsection{Coarea on the boundary}\label{subsec:boundary-coarea}

In this subsection we assume that $\partial\Omega$ is a Lipschitz hypersurface
(for instance $C^1$), so that there exists a unit normal vector
$n(x)$ for $\mathcal H^{n-1}$-almost every $x\in\partial\Omega$. For such $x$ we define the tangential gradient by
\begin{equation}\label{eq:tangential-gradient}
	\nabla_{\partial\Omega}\theta(x)
	:=
	\nabla\theta(x)-(\nabla\theta(x)\cdot n(x))\,n(x).
\end{equation}

\begin{remark}[Boundary transversality convention]\label{rem:boundary-transversality-notation}
	The geometric quantity that measures transversality between the fibers
	$\Sigma_t=\theta^{-1}(t)$ and the boundary $\partial\Omega$ is the norm of the tangential gradient
	\[
	|\nabla_{\partial\Omega}\theta(x)|.
	\]
	When $\partial\Omega$ is a Euclidean hypersurface and $n(x)$ denotes the outward unit normal, we shall occasionally use the equivalent notation
	\[
	|\nabla\theta\wedge n(x)|
	:=
	|\nabla_{\partial\Omega}\theta(x)|
	=
	\bigl(|\nabla\theta(x)|^2-(\nabla\theta(x)\cdot n(x))^2\bigr)^{1/2}.
	\]
	In particular, this quantity equals $|\nabla\theta(x)|\,|\sin\alpha(x)|$,
	where $\alpha(x)$ is the angle between $\nabla\theta(x)$ and the normal $n(x)$.
	Throughout the rest of the manuscript, the main notation will be
	$|\nabla_{\partial\Omega}\theta(x)|$, while the wedge notation will only be used as an abbreviation in concrete examples.
\end{remark}

The coarea formula applied to the restriction
$\theta|_{\partial\Omega}$ implies that, for every nonnegative measurable function
$F:\partial\Omega\to[0,\infty]$,
\begin{equation}\label{eq:coarea-boundary}
	\int_{\partial\Omega} F(x)\,|\nabla_{\partial\Omega}\theta(x)|\,
	\mathrm{d}\mathcal H^{n-1}(x)
	=
	\int_{\mathbb{R}}
	\left(
	\int_{\Sigma_t\cap\partial\Omega} F(x)\,\mathrm{d}\mathcal H^{n-2}(x)
	\right)\mathrm{d}t,
\end{equation}
and, equivalently,
\begin{equation}\label{eq:coarea-boundary-inv}
	\int_{\partial\Omega} F(x)\,\mathrm{d}\mathcal H^{n-1}(x)
	=
	\int_{\mathbb{R}}
	\left(
	\int_{\Sigma_t\cap\partial\Omega}
	\frac{F(x)}{|\nabla_{\partial\Omega}\theta(x)|}\,
	\mathrm{d}\mathcal H^{n-2}(x)
	\right)\mathrm{d}t,
\end{equation}
with the integrals understood as extended integrals. See, for instance,
\cite[Chapter~3]{EvansGariepy} for the coarea formula on rectifiable manifolds.

When the boundary is involved, the hypothesis \textbf{(H3)} introduces precisely a quantitative transversality condition excluding degenerate tangencies between the fibers $\Sigma_t$ and $\partial\Omega$. Under that hypothesis, the identities \eqref{eq:coarea-boundary}--\eqref{eq:coarea-boundary-inv} provide the starting point for the uniform geometric control of the intersections $\Sigma_t\cap\partial\Omega$.

%% file: 20_id_abs-sio.tex
\section{Exact reduction to the one-dimensional model}\label{sec:reduction-1d}

This section proves the first main result of the manuscript: the exact pushforward reduction of the truncated bilinear form on $\Omega_x\times\Omega_y$ to a one-dimensional truncated form on $\mathbb{R}\times\mathbb{R}$. The basic formulation belongs to the robust level of pushforward measures and disintegration fixed in Section~\ref{sec:measurable-core}; under additional absolute-continuity hypotheses, one recovers its realization in the Lebesgue layer.

The point of this section is structural: the synchronized geometric form is not compared with a one-dimensional model, but identified exactly with it, first at the level of pushforward measures and then, when appropriate, in terms of densities. This is the first station in the line of argument that culminates in the sparse recomposition of Section~\ref{sec:main-results}.

We shall use the generic notation $\theta$ to denote either of the two phases, $\phi$ or $\psi$, and write $\Omega_\theta$ for the corresponding domain. Thus, the quantities associated with a phase --- pushforward measure, density, fibers, and fiber operators --- are always understood with this convention.

\subsection{Reduction to the one-dimensional form in the pushforward measure}\label{subsec:reduction-nu}

In this subsection we formulate the exact reduction to the one-dimensional model at the intrinsic level of pushforward measures. The structural object in this regime is the relative density of the weighted pushforward measure with respect to the base pushforward measure, introduced in Definition~\ref{def:relative-density-pushforward}.

For measurable functions $F:\mathbb{R}\to\mathbb{C}$ and $G:\mathbb{R}\to\mathbb{C}$, we define the truncated bilinear form associated with the pushforward measures $\nu_\psi$ and $\nu_\phi$ by
\begin{equation}\label{eq:Lambda-eps-nu}
	\Lambda_\varepsilon^\nu(F,G)
	:=
	\int_{\mathbb{R}}\int_{\mathbb{R}}
	\mathbf{1}_{\{|s-t|>\varepsilon\}}\,k(s,t)\,F(t)\,G(s)\,
	\mathrm{d}\nu_\psi(t)\,\mathrm{d}\nu_\phi(s),
\end{equation}
whenever the integral is absolutely convergent.

\begin{theorem}[Exact pushforward reduction]\label{thm:exact-pushforward-reduction}
	Assume that
	\[
	K(x,y)=k(\phi(x),\psi(y))
	\qquad\text{for almost every }(x,y)\in\Omega_x\times\Omega_y,
	\]
	where $k$ satisfies \textup{(Hk1)}. Fix $\varepsilon>0$, and let
	$f\in L^\infty(\Omega_y)$ and $g\in L^\infty(\Omega_x)$ have compact support in
	$\Omega_y$ and $\Omega_x$, respectively. Then the truncated geometric form admits the exact representation
	\begin{equation}\label{eq:exact-pushforward-reduction}
		\langle T_\varepsilon f,g\rangle
		=
		\Lambda_\varepsilon^\nu(Q_{\psi,f},Q_{\phi,g}).
	\end{equation}
	In particular, the synchronized problem on $\Omega_x\times\Omega_y$ reduces exactly to a one-dimensional singular form in the level variables.

	If, in addition,
	\[
	\nu_\psi\ll \mathrm{d}t,
	\qquad
	\nu_\phi\ll \mathrm{d}s,
	\]
	then the previous identity takes the Lebesgue form
	\begin{equation}\label{eq:exact-pushforward-reduction-lebesgue}
		\langle T_\varepsilon f,g\rangle
		=
		\int_{\mathbb{R}}\int_{\mathbb{R}}
		\mathbf{1}_{\{|s-t|>\varepsilon\}}\,k(s,t)\,
		w_{\psi,f}(t)\,w_{\phi,g}(s)\,
		\mathrm{d}t\,\mathrm{d}s.
	\end{equation}
\end{theorem}

\begin{proof}
	Define
	\[
	H(x,y):=
	\mathbf{1}_{\{|\phi(x)-\psi(y)|>\varepsilon\}}\,
	k(\phi(x),\psi(y))\,f(y)\,g(x).
	\]
	The proof has two steps: first we obtain the exact identity at the level of pushforward measures and then, under absolute continuity, we rewrite it in the Lebesgue layer.

	\emph{Step 1: measure-level pushforward identity.}
	On the set
	\[
	\{(x,y)\in\Omega_x\times\Omega_y:\ |\phi(x)-\psi(y)|>\varepsilon\}
	\]
	the hypothesis \textup{(Hk1)} implies
	\[
	|k(\phi(x),\psi(y))|
	\le
	\frac{C_{k1}}{|\phi(x)-\psi(y)|}
	\le
	\frac{C_{k1}}{\varepsilon}.
	\]
	Therefore,
	\[
	|H(x,y)|
	\le
	\frac{C_{k1}}{\varepsilon}\,|f(y)|\,|g(x)|.
	\]
	Since $f$ and $g$ are bounded and compactly supported, the right-hand side belongs to
	$L^1(\Omega_x\times\Omega_y)$. In particular,
	\begin{equation}\label{eq:fubini-justified}
		\int_{\Omega_x}\int_{\Omega_y} |H(x,y)|\,\mathrm{d}y\,\mathrm{d}x<\infty,
		\qquad
		\int_{\Omega_x}\int_{\Omega_y} H\,\mathrm{d}y\,\mathrm{d}x
		=
		\int_{\Omega_y}\int_{\Omega_x} H\,\mathrm{d}x\,\mathrm{d}y.
	\end{equation}

	Fix $x\in\Omega_x$ and define
	\[
	\zeta_x(t)
	:=
	\mathbf{1}_{\{|\phi(x)-t|>\varepsilon\}}\,k(\phi(x),t).
	\]
	Because of the truncation, $\zeta_x$ is Borel measurable and bounded by $C_{k1}/\varepsilon$. Applying \eqref{eq:pushforward-relative-density} with $\theta=\psi$ and $h=f$, we obtain
	\begin{equation}\label{eq:coarea-justified-relative}
		\int_{\Omega_y}
		\mathbf{1}_{\{|\phi(x)-\psi(y)|>\varepsilon\}}\,
		k(\phi(x),\psi(y))\,f(y)\,\mathrm{d}y
		=
		\int_{\mathbb{R}}
		\mathbf{1}_{\{|\phi(x)-t|>\varepsilon\}}\,
		k(\phi(x),t)\,Q_{\psi,f}(t)\,\mathrm{d}\nu_\psi(t).
	\end{equation}
	Substituting this identity into the definition of $\langle T_\varepsilon f,g\rangle$, we get
	\begin{equation}\label{eq:pre-reduction-iterated-relative}
		\langle T_\varepsilon f,g\rangle
		=
		\int_{\Omega_x}
		\left(
		\int_{\mathbb{R}}
		\mathbf{1}_{\{|\phi(x)-t|>\varepsilon\}}\,
		k(\phi(x),t)\,Q_{\psi,f}(t)\,\mathrm{d}\nu_\psi(t)
		\right)
		g(x)\,\mathrm{d}x.
	\end{equation}
	Applying \eqref{eq:pushforward-relative-density} once more, now with $\theta=\phi$ and $h=g$, to the integrand in $x$, we obtain
	\begin{equation}\label{eq:reduction-identity-nu}
		\langle T_\varepsilon f,g\rangle
		=
		\int_{\mathbb{R}}\int_{\mathbb{R}}
		\mathbf{1}_{\{|s-t|>\varepsilon\}}\,k(s,t)\,
		Q_{\psi,f}(t)\,Q_{\phi,g}(s)\,
		\mathrm{d}\nu_\psi(t)\,\mathrm{d}\nu_\phi(s).
	\end{equation}
	By the definition of $\Lambda_\varepsilon^\nu$, this identity is exactly \eqref{eq:exact-pushforward-reduction}.

	\emph{Step 2: Lebesgue realization.}
	Assume now that
	\[
	\nu_\psi\ll \mathrm{d}t,
	\qquad
	\nu_\phi\ll \mathrm{d}s.
	\]
	We write
	\[
	\mathrm{d}\nu_\psi(t)=w_\psi(t)\,\mathrm{d}t,
	\qquad
	\mathrm{d}\nu_\phi(s)=w_\phi(s)\,\mathrm{d}s.
	\]
	Since, by Lemma~\ref{lem:nu-theta-h-ac},
	\[
	\nu_{\psi,f}\ll \nu_\psi,
	\qquad
	\nu_{\phi,g}\ll \nu_\phi,
	\]
	it follows that
	\[
	\nu_{\psi,f}\ll \mathrm{d}t,
	\qquad
	\nu_{\phi,g}\ll \mathrm{d}s.
	\]
	Thus we may write
	\[
	\mathrm{d}\nu_{\psi,f}(t)=w_{\psi,f}(t)\,\mathrm{d}t,
	\qquad
	\mathrm{d}\nu_{\phi,g}(s)=w_{\phi,g}(s)\,\mathrm{d}s.
	\]
	Moreover,
	\[
	\mathrm{d}\nu_{\psi,f}(t)
	=
	Q_{\psi,f}(t)\,\mathrm{d}\nu_\psi(t)
	=
	Q_{\psi,f}(t)\,w_\psi(t)\,\mathrm{d}t
	=
	w_{\psi,f}(t)\,\mathrm{d}t,
	\]
	and, analogously,
	\[
	\mathrm{d}\nu_{\phi,g}(s)
	=
	Q_{\phi,g}(s)\,\mathrm{d}\nu_\phi(s)
	=
	w_{\phi,g}(s)\,\mathrm{d}s.
	\]
	Substituting these identities into \eqref{eq:reduction-identity-nu}, we obtain \eqref{eq:exact-pushforward-reduction-lebesgue}.
\end{proof}

\begin{remark}[Robust level and coarea antecedent]
	The formulation \eqref{eq:reduction-identity-nu} is the robust level of the exact reduction: it involves only pushforward measures, weighted pushforward measures, and relative Radon--Nikodym derivatives with respect to $\nu_\phi$ and $\nu_\psi$. In particular, this layer does not require absolute continuity with respect to Lebesgue measure. The idea of reducing a multidimensional integral equation to a lower-dimensional problem by means of factorization and coarea formulae appears explicitly in Maz'ya's work on integral equations related to the coarea formula \cite{Mazya2007CoareaIE}; here it is used in a bilinear and truncated form, adapted to the synchronized framework of the manuscript.
\end{remark}

\subsection{Lebesgue formulation}\label{subsec:leb-formulation}

In this subsection we isolate the Lebesgue consequence of the exact reduction. Although this formulation already appears in Theorem~\ref{thm:exact-pushforward-reduction}, it is useful to record it as a corollary because it is the version used later to connect with the one-dimensional sparse input.

\begin{corollary}\label{cor:reduction-lebesgue}
	Under the hypotheses of Theorem~\ref{thm:exact-pushforward-reduction}, and assuming in addition that
	\[
	\nu_\psi\ll \mathrm{d}t,
	\qquad
	\nu_\phi\ll \mathrm{d}s,
	\]
	one has
	\begin{equation}\label{eq:theorem-A-lebesgue}
		\langle T_\varepsilon f,g\rangle
		=
		\int_{\mathbb{R}}\int_{\mathbb{R}}
		\mathbf{1}_{\{|s-t|>\varepsilon\}}\,k(s,t)\,
		w_{\psi,f}(t)\,w_{\phi,g}(s)\,
		\mathrm{d}t\,\mathrm{d}s.
	\end{equation}
\end{corollary}

\begin{proof}
	This is precisely the Lebesgue realization \eqref{eq:exact-pushforward-reduction-lebesgue} of Theorem~\ref{thm:exact-pushforward-reduction}.
\end{proof}

\begin{remark}[On the absolute-continuity hypothesis]
	The supplementary absolute-continuity hypothesis in the preceding corollary is not part of the robust level of the reduction, but rather of its Lebesgue rewriting. In the geometric submersion regimes covered by \textbf{(H1)} of Section~\ref{sec:measurable-core}, this absolute continuity holds locally on the corresponding level interval, and the densities $w_\psi$ and $w_\phi$ are given by the coarea formula.
\end{remark}

%% file: 25_sparse_abs-sio.tex
\section{Sparse domination in one dimension}\label{sec:sparse-1d}

This section provides the one-dimensional sparse input used in Theorem~\ref{thm:main-recomposition}. It is applied after the exact reduction of Theorem~\ref{thm:exact-pushforward-reduction} and produces the transferred output for the geometric form in the effective regime. The architecture is deliberately bipartite. On the one hand, the smoothed family of truncations
\[
(T^{1D}_{\varepsilon,\mathrm{sm}}F)(s)
=
\int_{\mathbb{R}} k(s,t)\,\chi\!\left(\frac{|s-t|}{\varepsilon}\right)F(t)\,\mathrm{d}t
\]
is inserted into the framework of Calder\'on--Zygmund operators with Dini regularity and treated through a bibliographical sparse-dual domination result. On the other hand, the difference between the hard truncation and the smooth truncation is controlled directly by the Hardy--Littlewood maximal operator and then reabsorbed into sparse form. In particular, this section does not reopen either the coarea reduction or the maximal-truncation and principal-value capsule reserved for Section~\ref{sec:pv-maximal-capsule}.

\subsection{Sparse families and sparse forms in \texorpdfstring{$\mathbb{R}$}{R}}\label{subsec:sparse-families-1d}

\begin{definition}[$\eta$-sparse]\label{def:sparse-eta}
Let $0<\eta<1$. We say that a finite or countable family $\mathcal{S}$ of intervals in $\mathbb{R}$ is $\eta$-sparse if for each $I\in\mathcal{S}$ there exists a measurable set $E_I\subset I$ such that the sets $\{E_I\}_{I\in\mathcal{S}}$ are pairwise disjoint and
\[
|E_I|\ge \eta |I|
\qquad 	\text{for every }I\in\mathcal{S}.
\]
\end{definition}

\begin{definition}[Lebesgue sparse form]\label{def:sparse-form-leb}
Let $\mathcal{S}$ be an $\eta$-sparse family of intervals in $\mathbb{R}$. For locally integrable functions $F,G$ we define
\begin{equation*}\label{eq:sparse-form-leb}
\Lambda_{\mathcal{S}}(F,G)
:=
\sum_{I\in\mathcal{S}}
\langle |F| \rangle_I
\langle |G| \rangle_I
|I|,
\end{equation*}
where
\[
\langle H\rangle_I
:=
\frac{1}{|I|}\int_I H(t)\,\mathrm{d}t.
\]

\end{definition}

\begin{definition}[Hardy--Littlewood maximal operator]\label{def:hardy-littlewood-maximal}
For every locally integrable function $H$ on $\mathbb{R}$ we define
\begin{equation*}\label{eq:hardy-littlewood-maximal}
M H(x)
:=
\sup_{I\ni x}\frac{1}{|I|}\int_I |H(t)|\,\mathrm{d}t,
\end{equation*}
where the supremum is taken over all intervals $I\subset\mathbb{R}$ containing $x$.
\end{definition}

\begin{lemma}[$L^r\times L^{r'}$ bound for sparse forms]\label{lem:sparse-form-Lp}
Let $\mathcal{S}$ be an $\eta$-sparse family of intervals in $\mathbb{R}$, with $0<\eta<1$. Then, for every $1<r<\infty$ and every pair of locally integrable functions $F,G$,
\begin{equation}\label{eq:sparse-form-Lp}
\Lambda_{\mathcal{S}}(F,G)
\lesssim_{r,\eta}
\|F\|_{L^r(\mathbb{R})}\,\|G\|_{L^{r'}(\mathbb{R})},
\qquad \frac1r+\frac1{r'}=1.
\end{equation}
\end{lemma}

\begin{proof}
For each \(I\in\mathcal{S}\) and almost every \(x\in E_I\subset I\), one has
\[
\langle |F| \rangle_I
\le
M(|F|)(x),
\qquad
\langle |G| \rangle_I
\le
M(|G|)(x),
\]
because the interval \(I\) appears in the definition of the Hardy--Littlewood maximal operator evaluated at \(x\). Since also \(|E_I|\ge \eta |I|\), we obtain
\[
\langle |F| \rangle_I
\langle |G| \rangle_I
|I|
\le
\eta^{-1}
\int_{E_I} M(|F|)(x)\,M(|G|)(x)\,\mathrm{d}x.
\]
Summing over \(I\in\mathcal{S}\) and using that the sets \(E_I\) are pairwise disjoint, we get
\[
\Lambda_{\mathcal{S}}(F,G)
\le
\eta^{-1}
\int_{\mathbb{R}} M(|F|)(x)\,M(|G|)(x)\,\mathrm{d}x.
\]
By H\"older's inequality and the boundedness of \(M\) on \(L^r(\mathbb{R})\) and \(L^{r'}(\mathbb{R})\), we conclude that
\[
\Lambda_{\mathcal{S}}(F,G)
\lesssim_{r,\eta}
\|M(|F|)\|_{L^r(\mathbb{R})}\,
\|M(|G|)\|_{L^{r'}(\mathbb{R})}
\lesssim_{r,\eta}
\|F\|_{L^r(\mathbb{R})}\,
\|G\|_{L^{r'}(\mathbb{R})}.
\]
This proves \eqref{eq:sparse-form-Lp}.
\end{proof}

\subsection{The one-dimensional truncated form}\label{subsec:1d-truncated-form}

By Corollary~\ref{cor:reduction-lebesgue}, the truncated geometric form is rewritten in terms of a one-dimensional bilinear form. We therefore fix as the central object of this section the form
\begin{equation}\label{eq:Lambda-eps-leb}
\Lambda_\varepsilon(F,G)
:=
\iint_{|t-s|>\varepsilon}
k(s,t)\,F(t)\,G(s)\,\mathrm{d}t\,\mathrm{d}s,
\end{equation}
initially defined for bounded compactly supported functions on $\mathbb{R}$. Equivalently,
\[
\Lambda_\varepsilon(F,G)
=
\int_{\mathbb{R}} (T^{1D}_\varepsilon F)(s)\,G(s)\,\mathrm{d}s,
\]
with $T^{1D}_\varepsilon$ given by \eqref{eq:T-hard}. Throughout this section we assume that the kernel $k$ satisfies the package \textbf{(Hk)} introduced in Subsection~\ref{subsec:kernel-1d}. In particular, we may use directly the size estimate \eqref{eq:k-size}, the Dini regularity \eqref{eq:kernel-smooth-1d-s}--\eqref{eq:kernel-smooth-1d-t}, and the uniform $L^2$ input for smooth truncations \eqref{eq:k-L2}.

\subsection{Sparse domination for the truncated form}\label{subsec:sparse-1d-main}

We first introduce the bibliographical black box that will be used for the smoothed family. The point to be verified in our context is that the family of kernels
\[
K_{\varepsilon,\mathrm{sm}}(s,t)
:=
k(s,t)\,\chi\!\left(\frac{|s-t|}{\varepsilon}\right),
\qquad s\neq t,
\]
inherits uniformly in $\varepsilon$ the size, Dini regularity, and $L^2$ boundedness required by the result of Ballesta--Yag\"ue--Conde--Alonso \cite{BallestaCondeAlonso2025DiniSparse}.

\begin{proposition}\label{prop:verify-sparse-black-box-1d}
Assume that the kernel $k$ satisfies \textbf{(Hk)}, and let $K_{\varepsilon,\mathrm{sm}}$ be the smoothed kernel above. Then the family $\{K_{\varepsilon,\mathrm{sm}}\}_{\varepsilon>0}$ satisfies, uniformly in $\varepsilon$, the following properties:
\begin{enumerate}
\item integral representation off the diagonal for $T^{1D}_{\varepsilon,\mathrm{sm}}$;
\item the size estimate
\begin{equation}\label{eq:kernel-size-soft-1d}
|K_{\varepsilon,\mathrm{sm}}(s,t)|
\lesssim
\frac{1}{|s-t|};
\end{equation}
\item a uniform Dini regularity condition in the first variable,
\begin{equation}\label{eq:kernel-dini-soft-1d}
|K_{\varepsilon,\mathrm{sm}}(s,t)-K_{\varepsilon,\mathrm{sm}}(s',t)|
\lesssim
\frac{\widetilde\omega\!\left(\frac{|s-s'|}{|s-t|}\right)}{|s-t|}
\end{equation}
whenever $2|s-s'|\le |s-t|$, where
\[
\widetilde\omega(u):=\omega(u)+u,
\qquad 0\le u\le 1,
\]
and, symmetrically, an analogous condition in the transposed variable,
\begin{equation}\label{eq:kernel-dini-soft-1d-transpose}
|K_{\varepsilon,\mathrm{sm}}(s,t)-K_{\varepsilon,\mathrm{sm}}(s,t')|
\lesssim
\frac{\widetilde\omega\!\left(\frac{|t-t'|}{|s-t|}\right)}{|s-t|}
\end{equation}
whenever $2|t-t'|\le |s-t|$;
\item the uniform boundedness
\begin{equation}\label{eq:L2-soft-uniform-1d}
\|T^{1D}_{\varepsilon,\mathrm{sm}}F\|_{L^2(\mathbb{R})}
\lesssim
\|F\|_{L^2(\mathbb{R})}.
\end{equation}
\end{enumerate}
Consequently, the family $\{T^{1D}_{\varepsilon,\mathrm{sm}}\}_{\varepsilon>0}$ enters uniformly the class of Calder\'on--Zygmund operators with Dini-smooth kernel to which Theorem~A of \cite{BallestaCondeAlonso2025DiniSparse} applies.
\end{proposition}

\begin{proof}
The integral representation off the diagonal is immediate from the definition. The size estimate \eqref{eq:kernel-size-soft-1d} follows from \eqref{eq:k-size} and from the fact that $0\le \chi\le 1$.

For the regularity in the first variable, we write
\[
K_{\varepsilon,\mathrm{sm}}(s,t)-K_{\varepsilon,\mathrm{sm}}(s',t)
=
\bigl(k(s,t)-k(s',t)\bigr)\chi\!\left(\frac{|s-t|}{\varepsilon}\right)
+
k(s',t)\Bigl(\chi\!\left(\frac{|s-t|}{\varepsilon}\right)-\chi\!\left(\frac{|s'-t|}{\varepsilon}\right)\Bigr).
\]
If $2|s-s'|\le |s-t|$, then also $|s'-t|\simeq |s-t|$. The first summand is bounded by \eqref{eq:kernel-smooth-1d-s}:
\[
\bigl|k(s,t)-k(s',t)\bigr|\,\chi\!\left(\frac{|s-t|}{\varepsilon}\right)
\le
\omega\!\left(\frac{|s-s'|}{|s-t|}\right)\frac{1}{|s-t|}.
\]
For the second summand we use the mean value theorem and the boundedness of $\chi'$:
\[
\left|\chi\!\left(\frac{|s-t|}{\varepsilon}\right)-\chi\!\left(\frac{|s'-t|}{\varepsilon}\right)\right|
\lesssim
\frac{||s-t|-|s'-t||}{\varepsilon}
\le
\frac{|s-s'|}{\varepsilon}.
\]
Moreover, this term can be nonzero only when $|s-t|\simeq \varepsilon$ or $|s'-t|\simeq \varepsilon$; hence, in that regime, one also has $\varepsilon\simeq |s-t|$. Combining this with \eqref{eq:k-size}, we obtain
\[
|k(s',t)|\left|\chi\!\left(\frac{|s-t|}{\varepsilon}\right)-\chi\!\left(\frac{|s'-t|}{\varepsilon}\right)\right|
\lesssim
\frac{1}{|s-t|}\,\frac{|s-s'|}{\varepsilon}
\lesssim
\frac{|s-s'|/|s-t|}{|s-t|}.
\]
This proves \eqref{eq:kernel-dini-soft-1d} with $\widetilde\omega(u)=\omega(u)+Cu$. Since
\[
\int_0^1 \widetilde\omega(u)\,\frac{\mathrm{d}u}{u}
=
\int_0^1 \omega(u)\,\frac{\mathrm{d}u}{u}+C<\infty,
\]
$\widetilde\omega$ remains a Dini modulus. The transposed estimate \eqref{eq:kernel-dini-soft-1d-transpose} is obtained in the same way using \eqref{eq:kernel-smooth-1d-t}.

By \eqref{eq:k-L2}, the family $\{T^{1D}_{\varepsilon,\mathrm{sm}}\}_{\varepsilon>0}$ is uniformly bounded on $L^2(\mathbb{R})$, which gives \eqref{eq:L2-soft-uniform-1d}. Thus the structural hypotheses required by Theorem~A of \cite{BallestaCondeAlonso2025DiniSparse} have been verified.
\end{proof}

\begin{corollary}\label{cor:sparse-soft-1d}
	Assume that the hypotheses of Proposition~\ref{prop:verify-sparse-black-box-1d} hold. Then there exists \(\eta\in(0,1)\), independent of \(\varepsilon\), such that for every pair of bounded, compactly supported, complex-valued functions \(F,G\) on \(\mathbb{R}\), there exists an \(\eta\)-sparse family \(\mathcal{S}_{\varepsilon,F,G}\) of intervals such that
	\begin{equation}\label{eq:sparse-soft-1d}
		\bigl|\langle T^{1D}_{\varepsilon,\mathrm{sm}}F,G\rangle\bigr|
		\lesssim
		\sum_{I\in\mathcal{S}_{\varepsilon,F,G}}
		\langle |F|\rangle_I\,\langle |G|\rangle_I\,|I|.
	\end{equation}
	The implicit constant is uniform in \(\varepsilon\).
\end{corollary}

\begin{proof}
	We first consider the case where \(F\) and \(G\) are nonnegative, bounded, and compactly supported. In that case, we apply Theorem~A of \cite{BallestaCondeAlonso2025DiniSparse} to each operator \(T^{1D}_{\varepsilon,\mathrm{sm}}\). Uniformity in \(\varepsilon\) follows from the fact that the size, Dini regularity, and \(L^2\)-boundedness parameters have been verified with constants independent of \(\varepsilon\) in Proposition~\ref{prop:verify-sparse-black-box-1d}.

	For general complex-valued functions, we write
	\[
	F=(\Re F)_+-(\Re F)_-+i(\Im F)_+-i(\Im F)_-,
	\]
	\[
	G=(\Re G)_+-(\Re G)_-+i(\Im G)_+-i(\Im G)_-.
	\]
	This decomposes \(\langle T^{1D}_{\varepsilon,\mathrm{sm}}F,G\rangle\) into a finite sum of terms of the form
	\[
	\langle T^{1D}_{\varepsilon,\mathrm{sm}}F_\alpha,G_\beta\rangle,
	\]
	where \(F_\alpha,G_\beta\) are nonnegative, bounded, and compactly supported, and satisfy pointwise
	\[
	0\le F_\alpha\le |F|,
	\qquad
	0\le G_\beta\le |G|.
	\]
	Applying the estimate already obtained in the nonnegative case to this finite collection of terms and reabsorbing the finite union of the resulting sparse families, we conclude \eqref{eq:sparse-soft-1d}.
\end{proof}

\begin{remark}\label{rem:sparse-black-box-source}
	The sparse domination of the smoothed family
	\[
	T^{1D}_{\varepsilon,\mathrm{sm}}
	\]
	is obtained here through Theorem~A of Ballesta--Yag\"ue--Conde--Alonso \cite{BallestaCondeAlonso2025DiniSparse}, once the relevant CZO-Dini hypotheses have been verified uniformly in \(\varepsilon\). The contribution of the present block therefore consists in this uniform verification for the smoothed family associated with our kernel and in its subsequent articulation with the hard--smooth error.
\end{remark}

We now turn to the hard--smooth error. This block is intrinsic to the manuscript, because it is the interface that allows one to return from the smoothed truncation to the hard truncation \eqref{eq:Lambda-eps-leb} without entering the later section devoted to maximal truncation and principal value.

\begin{lemma}\label{lem:hard-soft-pointwise-1d}
For every $\varepsilon>0$ and every locally integrable function $F$ on $\mathbb{R}$,
\begin{equation*}\label{eq:hard-soft-pointwise-1d}
|R^{1D}_\varepsilon F(s)|
=
\left|
\int_{\mathbb{R}} k(s,t)\Bigl(\mathbf 1_{\{|s-t|>\varepsilon\}}-\chi(|s-t|/\varepsilon)\Bigr)F(t)\,\mathrm{d}t
\right|
\lesssim
M F(s)
\end{equation*}
for almost every $s\in\mathbb{R}$, with a constant independent of $\varepsilon$.
\end{lemma}

\begin{proof}
The factor
\[
\mathbf 1_{\{|s-t|>\varepsilon\}}-\chi(|s-t|/\varepsilon)
\]
is supported where $\varepsilon<|s-t|<2\varepsilon$, because $\chi(r)=0$ for $0\le r\le 1$ and $\chi(r)=1$ for $r\ge 2$. By the size estimate \eqref{eq:k-size},
\[
|R^{1D}_\varepsilon F(s)|
\le
C_{k1}\int_{\varepsilon<|s-t|<2\varepsilon}\frac{|F(t)|}{|s-t|}\,\mathrm{d}t
\le
\frac{C_{k1}}{\varepsilon}\int_{|s-t|<2\varepsilon}|F(t)|\,\mathrm{d}t
\lesssim
M F(s).
\]
\end{proof}

\begin{lemma}\label{lem:hard-soft-error-bilinear-1d}
	For every pair of bounded, compactly supported, complex-valued functions \(F,G\),
	\begin{equation*}\label{eq:hard-soft-error-bilinear-1d}
		\bigl|\langle R^{1D}_\varepsilon F,G\rangle\bigr|
		\lesssim
		\langle M(|F|),|G|\rangle.
	\end{equation*}
	Consequently, there exists a sparse family \(\mathcal{S}'_{F,G}\) such that
	\begin{equation*}\label{eq:hard-soft-error-sparse-1d}
		\bigl|\langle R^{1D}_\varepsilon F,G\rangle\bigr|
		\lesssim
		\sum_{I\in\mathcal{S}'_{F,G}}
		\langle |F|\rangle_I\,\langle |G|\rangle_I\,|I|,
	\end{equation*}
	uniformly in \(\varepsilon\).
\end{lemma}

\begin{proof}
	By Lemma~\ref{lem:hard-soft-pointwise-1d},
	\[
	|R^{1D}_\varepsilon F(s)|
	\lesssim
	M(|F|)(s)
	\]
	for almost every \(s\in\mathbb{R}\). Integrating against \(|G|\), we obtain
	\[
	\bigl|\langle R^{1D}_\varepsilon F,G\rangle\bigr|
	\le
	\int_{\mathbb{R}} |R^{1D}_\varepsilon F(s)|\,|G(s)|\,\mathrm{d}s
	\lesssim
	\int_{\mathbb{R}} M(|F|)(s)\,|G(s)|\,\mathrm{d}s.
	\]
	The second inequality follows from the standard bilinear sparse domination of the Hardy--Littlewood maximal operator applied to \(|F|\) and \(|G|\).
\end{proof}

\begin{theorem}\label{thm:sparse-domination-1d}
	Assume that the hypotheses of Proposition~\ref{prop:verify-sparse-black-box-1d} hold. Then there exists \(\eta\in(0,1)\) such that, for every pair of bounded, compactly supported, complex-valued functions \(F,G\) on \(\mathbb{R}\), and for every \(\varepsilon>0\), there exists an \(\eta\)-sparse family \(\mathcal{S}_{\varepsilon,F,G}\) of intervals such that
	\begin{equation}\label{eq:sparse-domination-1d}
		\bigl|\Lambda_\varepsilon(F,G)\bigr|
		=
		\bigl|\langle T^{1D}_\varepsilon F,G\rangle\bigr|
		\lesssim
		\sum_{I\in\mathcal{S}_{\varepsilon,F,G}}
		\langle |F|\rangle_I\,\langle |G|\rangle_I\,|I|.
	\end{equation}
	The implicit constant is uniform in \(\varepsilon\).
\end{theorem}

\begin{proof}
	We decompose
	\[
	T^{1D}_\varepsilon
	=
	T^{1D}_{\varepsilon,\mathrm{sm}}+R^{1D}_\varepsilon.
	\]
	The contribution of the first term is controlled by Corollary~\ref{cor:sparse-soft-1d}, while the contribution of the second is controlled by Lemma~\ref{lem:hard-soft-error-bilinear-1d}. Adding the two bounds yields a sum of two sparse forms in terms of \(|F|\) and \(|G|\). Reabsorbing the finite union of the sparse families that appear, we conclude \eqref{eq:sparse-domination-1d}.
\end{proof}

\begin{remark}[Finite reabsorption of sparse families]\label{rem:finite-union-sparse}
	We shall use the following elementary stability property of sparseness: the union of finitely many $\eta_j$-sparse families can be reabsorbed, after modifying only the sparseness constant, into an $\eta$-sparse family, with $\eta>0$ depending only on the parameters $\eta_j$. All subsequent constants depending on $\eta$ incorporate this finite loss.
\end{remark}

\subsection{Transfer to the truncated geometric operator}\label{subsec:transfer-to-Teps}

We now combine Theorem~\ref{thm:sparse-domination-1d} with the exact Lebesgue identity of Corollary~\ref{cor:reduction-lebesgue}. At this stage one must carefully distinguish two levels of input: the preceding one-dimensional sparse theorem already allows complex-valued or signed functions, but it is still formulated for bounded and compactly supported inputs in the level variable. Thus, the transfer to the truncated geometric operator is obtained directly in the regime where the transferred fiber densities \(w_{\psi,f}\) and \(w_{\phi,g}\) belong to \(L^\infty_c(\mathbb{R})\).

\begin{corollary}\label{cor:sparse-transfer-Teps}
	Assume the hypotheses of Corollary~\ref{cor:reduction-lebesgue} and assume that the associated one-dimensional kernel satisfies the package \textbf{(Hk)} of Subsection~\ref{subsec:kernel-1d}. Assume moreover that, for the admissible pair of test functions \(f,g\), the transferred fiber densities satisfy
	\[
	w_{\psi,f},\,w_{\phi,g}\in L^\infty_c(\mathbb{R}).
	\]
	Then, for every \(\varepsilon>0\), there exists a sparse family \(\mathcal{S}_{\varepsilon,f,g}\) of intervals such that
	\begin{equation}\label{eq:sparse-domination-Teps}
		\bigl|\langle T_\varepsilon f,g\rangle\bigr|
		\lesssim
		\sum_{I\in\mathcal{S}_{\varepsilon,f,g}}
		\langle |w_{\psi,f}|\rangle_I\,\langle |w_{\phi,g}|\rangle_I\,|I|.
	\end{equation}
The implicit constant depends only on the analytic parameters of the one-dimensional model: the structural constants of the package \textbf{(Hk)}, the sparse-domination constant of the smoothed block, and the universal constants entering the hard--smooth error control.
\end{corollary}

\begin{proof}
	By Corollary~\ref{cor:reduction-lebesgue},
	\[
	\langle T_\varepsilon f,g\rangle
	=
	\Lambda_\varepsilon(w_{\psi,f},w_{\phi,g}).
	\]
	Since \(w_{\psi,f}\) and \(w_{\phi,g}\) are, by hypothesis, bounded, compactly supported, and possibly complex-valued functions, we may apply Theorem~\ref{thm:sparse-domination-1d} directly to them. We thus obtain
	\[
	\bigl|\Lambda_\varepsilon(w_{\psi,f},w_{\phi,g})\bigr|
	\lesssim
	\sum_{I\in\mathcal{S}_{\varepsilon,f,g}}
	\langle |w_{\psi,f}|\rangle_I\,\langle |w_{\phi,g}|\rangle_I\,|I|,
	\]
	which is exactly \eqref{eq:sparse-domination-Teps}.
\end{proof}

\begin{remark}[Scope of the one-dimensional sparse block]\label{rem:sparse-reduction-scope}\label{rem:sparse-1d-proof-scope}
	The sparse block of this section is closed by combining three ingredients: uniform sparse domination for the smoothed family, the hard--smooth comparison through the Hardy--Littlewood maximal operator, and the final transfer to the truncated geometric operator. The conclusion obtained is Corollary~\ref{cor:sparse-transfer-Teps}, in the effective regime where the transferred fiber densities belong to \(L^\infty_c(\mathbb{R})\).

	From this point on, the remaining part of the manuscript no longer consists in refining the one-dimensional sparse theory, but rather in structurally identifying those transferred inputs and controlling them from the geometry of the fibers and the regularity of the pushforwards. The extension to broader classes of densities would require an additional approximation and limiting argument, which lies outside the scope of this section.
\end{remark}

%% file: 26_pv_maximal_capsule.tex
\section{Maximal truncation and principal values in the one-dimensional model}\label{sec:pv-maximal-capsule}

This section collects complementary observations on the one-dimensional model fixed in Section~\ref{sec:sparse-1d}. Its role is to record the robustness of the formulation with respect to the cutoff, to compare precisely the hard and smoothed versions of the maximal truncation, and to delimit the scope of the observations concerning principal values.

The role of this block is therefore strictly complementary. The sparse domination of the smoothed operator and its transfer to the geometric operator were already closed in Section~\ref{sec:sparse-1d}; here we only add the robustness information needed to show that those formulations do not depend essentially on the choice of cutoff and to isolate the exact place of the maximal comparisons and principal-value observations within the chain of the manuscript.

\subsection{Robustness: cutoff independence}\label{subsec:pv-maximal-robust}

The following lemma quantifies the fact that different admissible choices of cutoff produce smooth truncations that differ, uniformly in $\varepsilon$, by a term controlled by the Hardy--Littlewood maximal operator.

\begin{lemma}[Equivalence of smooth truncations]\label{lem:cutoff-equivalence}
	Let $\chi_1,\chi_2:[0,\infty)\to[0,1]$ be measurable functions such that
	\[
	\chi_j(r)=0 \ \text{for } 0\le r\le 1,
	\qquad
	\chi_j(r)=1 \ \text{for } r\ge 2,
	\qquad j=1,2.
	\]
	For $j=1,2$, define
	\[
	(T^{1D,\chi_j}_{\varepsilon,\mathrm{sm}}f)(s)
	:=
	\int_{\mathbb{R}} k(s,t)\,\chi_j\!\left(\frac{|s-t|}{\varepsilon}\right)f(t)\,\mathrm{d}t.
	\]
	Then, for every $f\in L^\infty_c(\mathbb{R})$,
	\begin{equation}\label{eq:cutoff-equivalence}
		\sup_{\varepsilon>0}\bigl|(T^{1D,\chi_1}_{\varepsilon,\mathrm{sm}}f)(s)-(T^{1D,\chi_2}_{\varepsilon,\mathrm{sm}}f)(s)\bigr|
		\le C\,C_{k1}\,Mf(s)
	\end{equation}
	for almost every $s\in\mathbb{R}$, where $C>0$ is a universal constant.
\end{lemma}

\begin{proof}
	Let $m:=\chi_1-\chi_2$. Then $m$ is supported in $[1,2]$ and $|m|\le 1$. By \eqref{eq:k-size},
	\[
	\bigl|(T^{1D,\chi_1}_{\varepsilon,\mathrm{sm}}f)(s)-(T^{1D,\chi_2}_{\varepsilon,\mathrm{sm}}f)(s)\bigr|
	=
	\left|\int_{\mathbb{R}} k(s,t)\,m\!\left(\frac{|s-t|}{\varepsilon}\right)f(t)\,\mathrm{d}t\right|
	\le
	\int_{\varepsilon<|s-t|<2\varepsilon}\frac{C_{k1}}{|s-t|}\,|f(t)|\,\mathrm{d}t.
	\]
	Since $|s-t|\ge \varepsilon$ on the domain of integration,
	\[
	\int_{\varepsilon<|s-t|<2\varepsilon}\frac{C_{k1}}{|s-t|}\,|f(t)|\,\mathrm{d}t
	\le
	\frac{C_{k1}}{\varepsilon}\int_{|s-t|<2\varepsilon}|f(t)|\,\mathrm{d}t
	\le
	4C_{k1}\,Mf(s).
	\]
	Taking the supremum over $\varepsilon>0$ gives \eqref{eq:cutoff-equivalence}.
\end{proof}

\subsection{Hard--smooth comparison and maximal control}\label{subsec:pv-maximal-hard-soft}

This subsection records the maximal level of the comparison between hard and smooth truncations. The comparison at a fixed \(\varepsilon\) was already used in Section~\ref{sec:sparse-1d} to close the sparse domination; here the relevant object is the supremum in \(\varepsilon\), where the difference remains absorbable by the Hardy--Littlewood maximal operator.

\begin{lemma}[Maximal comparison: hard vs. smooth]\label{lem:hard-vs-soft-maximal-1d}
	For every $f\in L^\infty_c(\mathbb{R})$ one has
	\begin{equation}\label{eq:hard-vs-soft-maximal-1d}
		\sup_{\varepsilon>0}|T^{1D}_\varepsilon f(s)|
		\le
		\sup_{\varepsilon>0}|T^{1D}_{\varepsilon,\mathrm{sm}}f(s)|+C\,C_{k1}\,Mf(s)
	\end{equation}
	for almost every $s\in\mathbb{R}$, where $C>0$ is a universal constant depending only on the cutoff fixed in \eqref{eq:T-soft}.
\end{lemma}

\begin{proof}
	By the definition of the residual term \eqref{eq:R1d-def},
	\[
	|T^{1D}_\varepsilon f(s)|
	\le
	|T^{1D}_{\varepsilon,\mathrm{sm}}f(s)|+|R^{1D}_\varepsilon f(s)|.
	\]
	Moreover, by \eqref{eq:k-size} and by the support of the factor
	$\mathbf{1}_{\{|s-t|>\varepsilon\}}-\chi\!\left(\frac{|s-t|}{\varepsilon}\right)$,
	\[
	|R^{1D}_\varepsilon f(s)|
	\le
	\int_{\varepsilon<|s-t|<2\varepsilon}\frac{C_{k1}}{|s-t|}\,|f(t)|\,\mathrm{d}t
	\le
	4C_{k1}\,Mf(s).
	\]
	Therefore,
	\[
	|T^{1D}_\varepsilon f(s)|
	\le
	|T^{1D}_{\varepsilon,\mathrm{sm}}f(s)|+4C_{k1}\,Mf(s).
	\]
	Taking the supremum over $\varepsilon>0$ gives \eqref{eq:hard-vs-soft-maximal-1d}.
\end{proof}

\subsection{Principal-value observations}\label{subsec:pv-maximal-pv}

\begin{remark}[On the existence of principal values]\label{rem:pv-existence}
	The maximal control recorded above does not by itself imply the existence of principal values for the hard truncations \eqref{eq:T-hard} or for the smooth truncations \eqref{eq:T-soft}. Principal-value convergence requires additional hypotheses, which must be specified at the point where such an assertion is invoked.
\end{remark}

\subsection{Output toward the geometric control of fibers}\label{subsec:pv-maximal-geometric-output}

The preceding observations show that the one-dimensional formulation used in the sparse transfer is stable with respect to the choice of cutoff and compatible with the maximal comparisons recorded in this section. In particular, the output obtained in Section~\ref{sec:sparse-1d} can be reincorporated into the original geometric problem without depending on a specific realization of the truncation.

By Corollary~\ref{cor:reduction-lebesgue}, the singular-integral part of the geometric problem has already been identified exactly with a one-dimensional bilinear form. The content of the present block is to record that admissible changes of cutoff and the passage from smooth truncations to hard truncations at the maximal level introduce only errors controlled by the Hardy--Littlewood maximal operator in the level variable.

Thus the remaining task becomes geometric-analytic again, now on the fibers of \(\theta\). The singular-integral component has already been fixed in the one-dimensional model; the next step is to exploit geometrically the fiber quantities introduced in the preliminaries and reassembled in Section~\ref{sec:reduction-1d}.

%% file: 27_pushforward_lp_bridge.tex
\section{The pushforward operator bridge in the Lebesgue layer}\label{sec:pushforward-lp-bridge}

Theorem~\ref{thm:exact-pushforward-reduction} performs the exact reduction of the truncated geometric operator to a one-dimensional form in the level variable, and Section~\ref{sec:sparse-1d} fixes the sparse input for that reduced form. This section isolates the intermediate step that converts pushforward information in the Lebesgue layer into an abstract boundedness criterion for the transferred operator. It is the operator-theoretic bridge between the exact reduction and the recomposition of Theorem~\ref{thm:main-recomposition}.

We shall work here under the hypothesis \(\nu_\theta\ll \mathrm{d}t\). In the notation of the preliminaries, \(Q_{\theta,h}\) denotes the relative pushforward density with respect to \(\nu_\theta\), and we write
\[
A_\theta h:=w_{\theta,h}=w_\theta\,Q_{\theta,h}.
\]

\subsection{The pushforward operator in the Lebesgue layer}\label{subsec:pushforward-leb-operator}

We now pass from the robust level formulated in terms of pushforward measures to its realization in the Lebesgue layer. The information transported by the exact reduction can be expressed through the operator
\[
h\longmapsto A_\theta h,
\]
defined on functions in the original space and taking values in the level variable.

\begin{definition}[Pushforward operator in the Lebesgue layer]\label{def:pushforward-leb-operator}
	Let \(\theta:\Omega\to\mathbb{R}\) be a measurable function, and assume that the pushforward measure \(\nu_\theta=\theta_\#(\mu\llcorner \Omega)\) is absolutely continuous with respect to Lebesgue measure, with density \(w_\theta\in L^1_{\mathrm{loc}}(\mathbb{R})\). For \(h\in L^1_{\mathrm{loc}}(\Omega,\mu)\), we define \(A_\theta h\) as the density with respect to Lebesgue measure of the signed pushforward measure
	\[
	\nu_{\theta,h}:=\theta_\#(h\,\mu\llcorner\Omega),
	\]
	whenever this measure is absolutely continuous.
\end{definition}

Definition~\ref{def:pushforward-leb-operator} allows us to rewrite the exact identity of Section~\ref{sec:reduction-1d} in operator form. Under the hypotheses in which \eqref{eq:Q-theta-h-lebesgue} and Remark~\ref{rem:Q-vs-M} hold, this definition coincides with
\[
A_\theta h(t)=w_{\theta,h}(t)=w_\theta(t)\,Q_{\theta,h}(t)
\]
for almost every \(t\) such that \(0<w_\theta(t)<\infty\). In particular, in the absolutely continuous regime, the operator \(A_\theta\) coincides with the usual pushforward density, and this density in turn factorizes into the geometric level density \(w_\theta\) and the relative density \(Q_{\theta,h}\).

\begin{remark}[Operator-theoretic reading in the Lebesgue layer]\label{rem:pushforward-leb-operator-well-defined}\label{rem:pushforward-operator-reduction}
	Under the hypothesis \(\nu_\theta\ll \mathrm{d}t\) fixed at the beginning of the section, Lemma~\ref{lem:nu-theta-h-ac} guarantees that, for every \(h\in L^1_{\mathrm{loc}}(\Omega,\mu)\), the weighted pushforward measure \(\nu_{\theta,h}\) is also absolutely continuous with respect to Lebesgue measure. Consequently, Definition~\ref{def:pushforward-leb-operator} is effective in the regime considered here. Moreover, under the hypotheses of Corollary~\ref{cor:reduction-lebesgue}, the identity \eqref{eq:theorem-A-lebesgue} can be rewritten as
	\[
		\langle T_\varepsilon f,g\rangle
		=
		\Lambda_\varepsilon\bigl(A_\psi f,A_\phi g\bigr).
	\]
	Thus, in the Lebesgue layer, the exact reduction is expressed entirely in terms of \(A_\psi f\) and \(A_\phi g\). The remaining step is to control these operators in natural Lebesgue spaces.
\end{remark}

\subsection{An abstract \texorpdfstring{\(L^p\)}{Lp} bridge from the control of \texorpdfstring{\(w_\theta\)}{w theta}}\label{subsec:pushforward-lp-criterion}

This subsection closes the operator bridge in the Lebesgue layer: a uniform bound for \(w_\theta\) yields \(L^p\) control of \(A_\theta\), and hence the effective input of Section~\ref{sec:sparse-1d} is reduced to a verifiable geometric condition on the pushforward density.

\begin{proposition}[Abstract \texorpdfstring{\(L^p\)}{Lp} bridge]\label{prop:pushforward-linfty-to-lp}
	Let \(1<p<\infty\). Assume that
	\[
	\nu_\theta\ll \mathrm{d}t,
	\qquad
	w_\theta\in L^\infty(\mathbb{R}).
	\]
	Then, for every function \(h\in L^1(\Omega)\cap L^p(\Omega)\), one has
	\begin{equation}\label{eq:pushforward-linfty-to-lp}
		\|A_\theta h\|_{L^p(\mathbb{R})}
		\le
		\|w_\theta\|_{L^\infty(\mathbb{R})}^{1/p'}\,
		\|h\|_{L^p(\Omega)}.
	\end{equation}
	Moreover, the map
	\[
	h\mapsto A_\theta h
	\]
	initially defined on \(L^1(\Omega)\cap L^p(\Omega)\), extends uniquely to a bounded linear operator
	\[
	A_\theta:L^p(\Omega)\to L^p(\mathbb{R}),
	\]
	which still satisfies \eqref{eq:pushforward-linfty-to-lp}.
\end{proposition}
\begin{proof}
	We first fix \(h\in L^1(\Omega)\cap L^p(\Omega)\). By the definition of \(A_\theta h\) and by \eqref{eq:wtheta-times-Mtilde},
	\[
	A_\theta h(t)=w_\theta(t)\,\widetilde M_\theta h(t)
	\]
	for almost every \(t\) with \(0<w_\theta(t)<\infty\). Applying \eqref{eq:holder-fiber}, we obtain
	\[
	|A_\theta h(t)|
	=
	\left|
	\int_{\Sigma_t}\frac{h(x)}{|\nabla\theta(x)|}\,\mathrm{d}\mathcal H^{n-1}(x)
	\right|
	\le
	w_\theta(t)^{1/p'}\,M_\theta h(t)
	\]
	for almost every such \(t\). At points where \(w_\theta(t)=0\) or \(w_\theta(t)=\infty\), the definition of \(\widetilde M_\theta\) gives no contribution. Consequently,
	\[
	|A_\theta h(t)|^p
	\le
	\|w_\theta\|_{L^\infty(\mathbb{R})}^{p/p'}\,|M_\theta h(t)|^p
	\]
	for almost every \(t\in\mathbb{R}\). Integrating and using \eqref{eq:Mtheta-Lp-root}, it follows that
	\[
	\|A_\theta h\|_{L^p(\mathbb{R})}
	\le
	\|w_\theta\|_{L^\infty(\mathbb{R})}^{1/p'}\,
	\|M_\theta h\|_{L^p(\mathbb{R})}
	=
	\|w_\theta\|_{L^\infty(\mathbb{R})}^{1/p'}\,
	\|h\|_{L^p(\Omega)}.
	\]
	This proves \eqref{eq:pushforward-linfty-to-lp} on \(L^1(\Omega)\cap L^p(\Omega)\).

	Linearity on this core is immediate from the linearity of the weighted pushforward measure:
	\[
	\nu_{\theta,\alpha h_1+\beta h_2}
	=
	\alpha\,\nu_{\theta,h_1}+\beta\,\nu_{\theta,h_2},
	\]
	and therefore
	\[
	A_\theta(\alpha h_1+\beta h_2)
	=
	\alpha\,A_\theta h_1+\beta\,A_\theta h_2
	\]
	almost everywhere.

	Since \(L^1(\Omega)\cap L^p(\Omega)\) is dense in \(L^p(\Omega)\), the bound \eqref{eq:pushforward-linfty-to-lp} allows one to extend \(A_\theta\) uniquely by continuity to all of \(L^p(\Omega)\), preserving the same operator norm. We continue to denote this extension by \(A_\theta\).
\end{proof}

For the application to Section~\ref{sec:sparse-1d}, we also record the \(L^\infty_c\) version.

\begin{corollary}[Bounded compactly supported inputs]\label{cor:pushforward-linfty-compact-support}
	Assume that
	\[
	\nu_\theta\ll \mathrm{d}t,
	\qquad
	w_\theta\in L^\infty(\mathbb{R}).
	\]
	Let \(h\in L^\infty(\Omega)\) have compact support. Then
	\begin{equation*}\label{eq:pushforward-linfty-compact-support}
		A_\theta h\in L^\infty_c(\mathbb{R})
		\qquad\text{and}\qquad
		\|A_\theta h\|_{L^\infty(\mathbb{R})}
		\le
		\|w_\theta\|_{L^\infty(\mathbb{R})}\,\|h\|_{L^\infty(\Omega)}.
	\end{equation*}
	Moreover,
	\[
	\operatorname{supp}(A_\theta h)\subset \theta(\operatorname{supp} h).
	\]
\end{corollary}

\begin{proof}
	By \eqref{eq:wtheta-rho-def},
	\[
	A_\theta h(t)
	=
	\int_{\Sigma_t}\frac{h(x)\,\mathbf 1_{\{|\nabla\theta|>0\}}(x)}{|\nabla\theta(x)|}\,\mathrm{d}\mathcal H^{n-1}(x)
	\]
	for almost every \(t\). Hence
	\[
	|A_\theta h(t)|
	\le
	\|h\|_{L^\infty(\Omega)}
	\int_{\Sigma_t}\frac{\mathbf 1_{\{|\nabla\theta|>0\}}(x)}{|\nabla\theta(x)|}\,\mathrm{d}\mathcal H^{n-1}(x)
	=
	\|h\|_{L^\infty(\Omega)}\,w_\theta(t),
	\]
	which gives the \(L^\infty\) bound.

	For the support, let \(E\subset\mathbb{R}\) be a Borel set such that
	\[
	E\cap \theta(\operatorname{supp} h)=\varnothing.
	\]
	Then \(\theta^{-1}(E)\cap \operatorname{supp} h=\varnothing\), and by the definition of \(\nu_{\theta,h}\),
	\[
	\nu_{\theta,h}(E)
	=
	\int_{\theta^{-1}(E)} h(x)\,\mathrm{d}x
	=
	0.
	\]
	This implies that \(A_\theta h=0\) almost everywhere on \(E\). Thus
	\[
	\operatorname{supp}(A_\theta h)\subset \theta(\operatorname{supp} h).
	\]
	Since \(\operatorname{supp} h\) is compact and \(\theta\) is continuous in the geometric framework of the manuscript, \(\theta(\operatorname{supp} h)\) is compact, and therefore \(A_\theta h\in L^\infty_c(\mathbb{R})\).
\end{proof}

\begin{corollary}[Structural input for the sparse transfer]\label{cor:pushforward-structural-sparse-input}
	Assume that
	\[
	\nu_\psi\ll \mathrm{d}t,
	\qquad
	\nu_\phi\ll \mathrm{d}s,
	\qquad
	w_\psi\in L^\infty(\mathbb{R}),
	\qquad
	w_\phi\in L^\infty(\mathbb{R}).
	\]
	Then, for every pair of bounded, compactly supported, complex-valued functions \(f,g\) on \(\Omega\),
	\[
	A_\psi f\in L^\infty_c(\mathbb{R}),
	\qquad
	A_\phi g\in L^\infty_c(\mathbb{R}).
	\]
	In particular, the effective hypothesis for applying Corollary~\ref{cor:sparse-transfer-Teps} is automatically satisfied, and for every \(\varepsilon>0\) there exists a sparse family \(\mathcal S_{\varepsilon,f,g}\) of intervals such that
	\begin{equation}\label{eq:pushforward-structural-sparse-input}
		\bigl|\langle T_\varepsilon f,g\rangle\bigr|
		\lesssim
		\sum_{I\in\mathcal S_{\varepsilon,f,g}}
		\langle |A_\psi f|\rangle_I\,
		\langle |A_\phi g|\rangle_I\,|I|.
	\end{equation}
\end{corollary}

\begin{proof}
The first assertion is a direct application of Corollary~\ref{cor:pushforward-linfty-compact-support} with \(\theta=\psi\) and \(\theta=\phi\). Consequently, the effective hypothesis of Corollary~\ref{cor:sparse-transfer-Teps} is verified for \(A_\psi f\) and \(A_\phi g\), and \eqref{eq:pushforward-structural-sparse-input} is precisely its rewriting in the notation of this section.
\end{proof}

It remains to verify geometrically the condition
\[
w_\theta\in L^\infty(\mathbb{R})
\]
in concrete regimes.

%% file: 30_geom_abs-sio.tex
\section{Local recomposition in fiber language}\label{sec:fiber-analytic-objective}

This section begins the geometric part of the recomposition in Theorem~\ref{thm:main-recomposition}, using the fiber language fixed in the preliminaries. We rewrite the transferred output in terms of the expressions
\[
w_\theta\,\widetilde M_\theta h,
\]
record a first abstract closure principle based on fiber control, and obtain a local output in the uniform regime on level intervals.

\subsection{Formulation of the recomposition in fiber language}\label{subsec:fiber-analytic-formulation}

By the exact identity in the Lebesgue layer and the operator-theoretic reading of Section~\ref{sec:pushforward-lp-bridge},
\[
\langle T_\varepsilon f,g\rangle
=
\Lambda_\varepsilon\bigl(A_\psi f,A_\phi g\bigr).
\]
Under the dictionary fixed in the preliminaries,
\[
A_\theta h
=
w_\theta\,\widetilde M_\theta h
\]
whenever \(0<w_\theta<\infty\). Hence the geometric recomposition is formulated in terms of the weighted fiber expressions associated with \(\psi\) and \(\phi\).

\begin{proposition}[Closure principle by fiber control]\label{prop:fiber-control-suffices}
	Assume that, for some exponent \(1\le p_0<\infty\), there exists a constant \(C_{\mathrm{fib}}>0\) such that, for every measurable function \(h\) on \(\Omega\),
	\begin{equation*}\label{eq:fiber-control-bilinear}
		\|w_\theta\,\widetilde M_\theta h\|_{L^{p_0}(\mathbb{R})}
		\le
		C_{\mathrm{fib}}\,
		\|h\|_{L^{p_0}(\Omega)}
	\end{equation*}
	for \(\theta=\phi\) and for \(\theta=\psi\).
	Then every bilinear estimate for the one-dimensional model formulated in terms of
	\[
	F=A_\psi f,
	\qquad
	G=A_\phi g
	\]
	transfers immediately to the truncated geometric form \(\langle T_\varepsilon f,g\rangle\), with constant controlled by \(C_{\mathrm{fib}}\) and by the constants of the corresponding one-dimensional result.
\end{proposition}
\begin{proof}
	The assertion is a direct consequence of
	\[
	\langle T_\varepsilon f,g\rangle
	=
	\Lambda_\varepsilon\bigl(A_\psi f,A_\phi g\bigr)
	\]
	and of the identification
	\[
	A_\theta h=w_\theta\,\widetilde M_\theta h.
	\]
\end{proof}

\begin{remark}[Local scope of this section]\label{rem:fiber-analytic-roadmap}
	The output obtained here is local in the level parameter and will serve as an input for the modular recomposition of Section~\ref{sec:main-results}.
\end{remark}

\subsection{Local uniform regime and stable fiber control}\label{subsec:fiber-uniform-submersion}

The first regime in which the local recomposition becomes effective is the one where the geometry of the level sets remains quantitatively stable on a level interval. In this context, the weighted expression
\[
w_\theta\,\widetilde M_\theta h
\]
admits robust control in Lebesgue spaces.

This subsection records precisely that positive local mechanism. Its role is not yet to close the global recomposition, but rather to isolate the geometric block that will later feed the uniform consequence in Section~\ref{sec:main-results}.

\begin{corollary}[Uniform submersion on \texorpdfstring{$I_0$}{I0} \(\Longrightarrow\) local weighted control]\label{cor:wtheta-submersion-uniform}
	Assume that $\theta$ satisfies \textbf{(H1)} and \textbf{(H2)} on an open interval $I_0\subset\mathbb{R}$, so that the local representation \eqref{eq:wtheta-rho-chart} is available for almost every $t\in I_0$.
	Assume moreover that there exists a constant $C_{\mathrm{sub}}(I_0)>0$ such that
	\[
	w_\theta(t)\le C_{\mathrm{sub}}(I_0)
	\qquad\text{for almost every }t\in I_0.
	\]

	Then, for every $1<r<\infty$ and every $f\in L^r(\Omega)$,
	\begin{equation*}\label{eq:submersion-uniform-weighted}
		\|w_\theta\,\widetilde M_\theta f\|_{L^r(I_0)}
		\le
		C_{\mathrm{sub}}(I_0)^{1/r'}\,\|f\|_{L^r(\Omega)}.
	\end{equation*}
\end{corollary}

\begin{proof}
	By the identification
	\[
	A_\theta f=w_\theta\,\widetilde M_\theta f
	\]
	almost everywhere on \(I_0\), it is enough to estimate \(A_\theta f\) in \(L^r(I_0)\). For almost every \(t\in I_0\), H\"older's inequality on the fiber, \eqref{eq:holder-fiber}, gives
	\[
	|A_\theta f(t)|
	=
	\left|
	\int_{\Sigma_t}\frac{f(x)}{|\nabla\theta(x)|}\,\mathrm{d}\mathcal H^{n-1}(x)
	\right|
	\le
	w_\theta(t)^{1/r'}\,M_\theta f(t).
	\]
	Since \(w_\theta(t)\le C_{\mathrm{sub}}(I_0)\) for almost every \(t\in I_0\), it follows that
	\[
	|A_\theta f(t)|
	\le
	C_{\mathrm{sub}}(I_0)^{1/r'}\,M_\theta f(t)
	\qquad\text{for almost every }t\in I_0.
	\]
	Raising to the power \(r\), integrating on \(I_0\), and using \eqref{eq:Mtheta-Lp-root}, we obtain
	\begin{align*}
	\|w_\theta\,\widetilde M_\theta f\|_{L^r(I_0)}
	=
	\|A_\theta f\|_{L^r(I_0)}\le
	C_{\mathrm{sub}}(I_0)^{1/r'}\,\|M_\theta f\|_{L^r(I_0)} \le
	C_{\mathrm{sub}}(I_0)^{1/r'}\,& \|M_\theta f\|_{L^r(\mathbb{R})}\\
	& =
	C_{\mathrm{sub}}(I_0)^{1/r'}\,\|f\|_{L^r(\Omega)}.
	\end{align*}
\end{proof}

\begin{remark}[From local to global]\label{rem:wtheta-uniform-globalization}
	The preceding corollary is local in the level parameter: \textbf{(H1)} and \textbf{(H2)} were formulated in a tube \(I_0\), and therefore do not by themselves yield a global conclusion on all of \(\theta(\Omega)\). To obtain a global output of the form
	\[
	\|w_\theta\,\widetilde M_\theta f\|_{L^r(\mathbb{R})}
	\lesssim
	\|f\|_{L^r(\Omega)},
	\]
	one needs an additional finite-covering hypothesis for \(\theta(\Omega)\) by level intervals on which the uniform trivialization is available with controlled constants. That hypothesis is introduced in Section~\ref{sec:main-results}.
\end{remark}

\subsection{Quantified non-degeneracy and design of favorable phases}\label{subsec:quantified-nondegeneracy-design}

A natural way to build phases adapted to the reduction--recomposition scheme is to impose a differential lower bound of the form
\[
|\nabla \theta(x)|\ge \Gamma(\theta(x)),
\]
where \(\Gamma\) prescribes the minimal separation between level sets. This condition weakens the classical assumption of strict uniform submersion and translates the geometric control of the phase into a scalar profile on the level space.

Under this hypothesis, the coarea density satisfies
\[
w_\theta(t)
=
\int_{\Sigma_t}\frac{1}{|\nabla\theta|}\,\mathrm{d}\mathcal H^{n-1}
\le
\frac{\mathcal H^{n-1}(\Sigma_t)}{\Gamma(t)},
\qquad
\Sigma_t:=\{x\in\Omega:\theta(x)=t\},
\]
so the singular coarea factor splits into a geometric contribution, \(\mathcal H^{n-1}(\Sigma_t)\), and a scalar design contribution, \(\Gamma(t)\).

A convenient way to generate this type of phase is by reparametrization. Let \(\rho:\Omega\to I\subset\mathbb R\) be a Lipschitz phase satisfying
\[
|\nabla \rho(x)|\ge m(\rho(x))
\]
for some nonnegative function \(m\). Given a monotone function \(H:I\to J\) of class \(C^1\), define
\[
\theta(x):=H(\rho(x)).
\]
Then
\[
\nabla \theta(x)=H'(\rho(x))\,\nabla \rho(x),
\]
and hence
\[
|\nabla\theta(x)|
=
|H'(\rho(x))|\,|\nabla\rho(x)|
\ge
|H'(\rho(x))|\,m(\rho(x)).
\]
If \(H\) is invertible, this can be rewritten in terms of levels as
\[
|\nabla\theta(x)|
\ge
|H'(H^{-1}(\theta(x)))|\,m(H^{-1}(\theta(x))),
\]
so the effective design function is
\[
\Gamma(t):=
|H'(H^{-1}(t))|\,m(H^{-1}(t)).
\]

Conversely, if one wants to prescribe a target lower bound
\[
|\nabla\theta|\ge \Gamma(\theta),
\]
it is enough to choose \(H\) as a solution of the scalar ODE
\[
H'(s)\,m(s)=\Gamma(H(s)).
\]
In particular, if \(|\nabla\rho|=1\) almost everywhere, then \(m\equiv 1\) and the design equation reduces to
\[
H'(s)=\Gamma(H(s)).
\]

This produces several natural families of phases favorable to the uniform regime.

\paragraph{(i) Transverse coordinates.}
If \(\rho\) is a transverse coordinate with \(|\nabla\rho|=1\), for instance a signed distance in a tubular neighborhood or a vertical coordinate in a product-type region, then every phase of the form
\[
\theta=H\circ \rho
\]
satisfies
\[
|\nabla\theta|=|H'(\rho)|.
\]
Thus the lower bound \(|\nabla\theta|\ge \Gamma(\theta)\) is guaranteed by solving
\[
H'(s)=\Gamma(H(s)).
\]

\paragraph{(ii) Distance-to-the-boundary phases.}
If \(\rho(x)=d(x,\partial\Omega)\), then \(|\nabla\rho|=1\) almost everywhere, and therefore
\[
\theta(x)=H(d(x,\partial\Omega))
\]
again reduces the design problem to the scalar ODE above. This is particularly useful for constructing families with quantitative control near a boundary level without imposing, from the outset, a rigid uniform submersion on the whole image.

\paragraph{(iii) Radial phases.}
If \(\rho(x)=|x|\), then \(|\nabla\rho|=1\) away from the origin. Hence radial phases
\[
\theta(x)=H(|x|)
\]
fall into the same class. In this case, the geometry of the level sets is explicit, and both \(\mathcal H^{n-1}(\Sigma_t)\) and \(w_\theta(t)\) can often be computed by closed formulas. This class will later serve as a concrete laboratory in Section~\ref{sec:examples-regimes}.

\paragraph{(iv) General base phases.}
More generally, any Lipschitz phase \(\rho\) satisfying a lower bound
\[
|\nabla\rho|\ge m(\rho)
\]
can be corrected by reparametrization. The choice of \(H\) then solves
\[
H'(s)=\frac{\Gamma(H(s))}{m(s)}.
\]
This provides a flexible mechanism for improving a given phase while preserving the geometry of its levels.

\begin{remark}[Uniform profile and critical transition]\label{rem:gamma-design-uniform}\label{rem:gamma-design-critical-transition}
	If
	\[
	\Gamma(t)\ge c>0,
	\]
	then
	\[
	|\nabla\theta|\ge c,
	\qquad
	w_\theta(t)\le \frac{\mathcal H^{n-1}(\Sigma_t)}{c}.
	\]
	In particular, when the geometry of the fibers remains quantitatively controlled on a level interval, the condition \(|\nabla\theta|\ge \Gamma(\theta)\) recovers the local uniform regime of Corollary~\ref{cor:wtheta-submersion-uniform}. When \(\Gamma\) is no longer uniformly separated from zero, the splitting
	\[
	w_\theta(t)\le \frac{\mathcal H^{n-1}(\Sigma_t)}{\Gamma(t)}
	\]
	shows how degeneration in the separation between levels may be reflected in the pushforward density. This is the transition toward the critical regime of Section~\ref{sec:fiber-critical-regime}.
\end{remark}

\begin{remark}[Analytic source of critical profiles]\label{rem:gamma-lojasiewicz-source}
	The design scheme based on profiles \(\Gamma\) has a natural analytic source. In broad analytic classes, inequalities of \L{}ojasiewicz type provide quantitative lower bounds for the gradient near critical values; see, for instance, \cite{Chill2003,NguyenPham2022}. More precisely, if \(\theta\) is real analytic and \(x_0\) is a critical point with \(t_0=\theta(x_0)\), then there exist a neighborhood of \(x_0\), a constant \(C>0\), and an exponent \(\alpha\in[0,1)\) such that
	\[
	|\nabla \theta(x)| \ge C\,|\theta(x)-t_0|^{\alpha}.
	\]
	In the language of this section, this means that profiles of the form
	\[
	\Gamma(t)\sim |t-t_0|^{\alpha}
	\]
	arise naturally as quantitative laws of non-degeneracy. Combined with geometric control of the fibers, they provide a canonical source for the abstract critical profiles of Section~\ref{sec:fiber-critical-regime}.
\end{remark}

\paragraph{Taxonomy of phases and transition of regimes.}
The condition
\[
|\nabla \theta(x)| \ge \Gamma(\theta(x))
\]
distinguishes, at the design level, two behaviors. If \(\Gamma\) remains uniformly separated from zero on the level interval under consideration, one recovers the local uniform regime. When \(\Gamma\) weakens or degenerates near one or more levels, the pushforward density may develop singularities compatible with the critical regime.

%% file: 35_geom_abs-sio.tex
\section{Critical regime: localization and pullback weights}\label{sec:fiber-critical-regime}

This section isolates the output corresponding to the critical regime within the geometric recomposition of Theorem~\ref{thm:main-recomposition}. Unlike the uniform case of Section~\ref{sec:fiber-analytic-objective}, the loss of geometric stability near critical values no longer allows, in general, the fiber formulation to be closed by an unweighted estimate on the level space.

The central object is the pushforward density \(w_\theta\) associated with the phase \(\theta\). When quantitative non-degeneracy weakens near certain levels, \(w_\theta\) may develop singularities localized around the critical set \(V_\theta\). The purpose of this section is to record, in an abstract framework, that such singularities still admit a useful recomposition if one works with localization in the level space and with a pullback weight on the input space.

We first fix an abstract blow-up profile for \(w_\theta\) near \(V_\theta\) and the range of exponents for which this profile remains locally integrable. We then translate this information into localized estimates for \(w_\theta\,\widetilde M_\theta\). Thus, the critical regime enters the general architecture of the manuscript as a localized and weighted output, rather than as a uniform continuation of the previous case.

\subsection{Blow-up profile near the critical values}\label{subsec:critical-blowup-profile}

We now fix the minimal abstract hypothesis of the critical regime. We assume:
\begin{enumerate}
	\item the set $V_\theta\subset\mathbb{R}$ is finite;
	\item $w_\theta\in L^\infty_{\mathrm{loc}}(\mathbb{R}\setminus V_\theta)$;
	\item there exist $\delta_0>0$, a constant $C_{\theta,\beta}>0$, and an exponent $\beta\in[0,1)$ such that
	\begin{equation}\label{eq:wtheta-critical-blowup}
		w_\theta(t)
		\le
		C_{\theta,\beta}\,\operatorname{dist}(t,V_\theta)^{-\beta}
		\qquad
		\text{for almost every }t\text{ with }\operatorname{dist}(t,V_\theta)<\delta_0.
	\end{equation}
\end{enumerate}

The following proposition records the local integrability consequence of \eqref{eq:wtheta-critical-blowup}.

\begin{proposition}[Integrability of the critical profile]\label{prop:wtheta-critical-integrability}
	Assume that $V_\theta$ is finite and that \eqref{eq:wtheta-critical-blowup} holds for some $\beta\in[0,1)$.
	Then, for every exponent
	\begin{equation}\label{eq:wtheta-critical-integrability-range}
		1\le a<\frac{1}{\beta}
		\qquad\text{(convention: }1/0=\infty\text{)},
	\end{equation}
	one has
	\[
	w_\theta\in L^a_{\mathrm{loc}}
	\bigl(\{t\in\mathbb{R}:\operatorname{dist}(t,V_\theta)<\delta_0\}\bigr).
	\]
	In particular, for every
	\[
	1<r<1+\frac{1}{\beta},
	\]
	the function $w_\theta^{\,r-1}$ is locally integrable in a neighborhood of $V_\theta$.
\end{proposition}

\begin{proof}
	Since $V_\theta$ is finite, write
	\[
	V_\theta=\{\tau_1,\dots,\tau_N\}.
	\]
	By decreasing $\delta_0$ if necessary, we may assume that the intervals
	\[
	I_j:=(\tau_j-\delta_0,\tau_j+\delta_0),\qquad 1\le j\le N,
	\]
	are pairwise disjoint. Then, by \eqref{eq:wtheta-critical-blowup},
	\[
	\int_{\operatorname{dist}(t,V_\theta)<\delta_0} |w_\theta(t)|^a\,\mathrm{d}t
	\le
	C_{\theta,\beta}^a
	\sum_{j=1}^N \int_{I_j} |t-\tau_j|^{-a\beta}\,\mathrm{d}t.
	\]
	Each integral on the right-hand side is finite if and only if $a\beta<1$, that is, if $a<1/\beta$. This proves the first assertion. The second follows by applying the first one with $a=r-1$.
\end{proof}

Proposition~\ref{prop:wtheta-critical-integrability} fixes the natural exponent window of the critical regime and anticipates the range that will reappear in the final recomposition.

\subsection{Critical values, profiles of \texorpdfstring{\(w_\theta\)}{w theta}, and scope of the formulation}\label{subsec:critical-values-scope}

The presence of critical values must be interpreted carefully within the architecture of the critical regime. The mere presence of a critical level \(t\in V_\theta\) does not determine by itself either the quantitative profile of the pushforward density \(w_\theta\) or the nature of the final functional output. What is decisive is the effective behavior of \(w_\theta\) in a neighborhood of \(V_\theta\).

In particular, the condition
\[
V_\theta\neq\varnothing
\]
does not by itself imply the loss of a uniform output. As the examples in Section~\ref{sec:examples-regimes} show, critical levels may occur without blow-up of \(w_\theta\), logarithmic growth may appear, or a power profile of the form
\[
w_\theta(t)\lesssim \operatorname{dist}(t,V_\theta)^{-\beta}
\]
may emerge. The abstract hypothesis \eqref{eq:wtheta-critical-blowup} therefore functions as a sufficient quantitative envelope for localized recomposition, without aiming to describe universally the entire critical phenomenology.

Equivalently, the set \(V_\theta\) localizes the geometric obstruction, while the parameter \(\beta\) quantifies its severity through the behavior of \(w_\theta\). This is the logic organizing the present section:
\[
V_\theta \leadsto w_\theta \leadsto \beta.
\]

\begin{remark}\label{rem:critical-values-atomic-pushforward}
	The formulation in terms of the density \(w_\theta\) presupposes an absolutely continuous regime for the pushforward measure associated with \(\theta\). If there is a region of positive measure on which \(\theta\) is constant and \(\nabla\theta=0\), then the pushforward may acquire an atomic part. In that case, the density formulation with respect to Lebesgue measure is no longer the natural interface, and the robust level of the analysis returns to the pushforward measure itself.
\end{remark}

The concrete geometric mechanisms that may produce these critical exponents will be illustrated later in Section~\ref{sec:examples-regimes}.

\subsection{Analytic output for \texorpdfstring{\(w_\theta\,\widetilde M_\theta\)}{w theta M theta}}\label{subsec:critical-weighted-control}

In the critical regime, the weighted formulation inherited from \eqref{eq:wtheta-times-Mtilde} no longer leads, in general, to an unweighted bound on \(\Omega\). H\"older's inequality on the fiber, recorded in \eqref{eq:holder-fiber}, shows that the natural analytic cost of this degeneration is the appearance of the composite weight
\[
w_\theta(\theta(x))^{r-1}.
\]

Within the fiber framework fixed in the preliminaries, the critical loss is translated into the following basic weighted inequality.

\begin{proposition}[Localized control of the weighted expression]\label{prop:weighted-fiber-localized}
	Assume that \(\theta\) belongs to the geometric framework of the preliminaries, so that \eqref{eq:wtheta-times-Mtilde}, \eqref{eq:holder-fiber}, and \eqref{eq:coarea} hold. Fix $1<r<\infty$ and let $E\subset\mathbb{R}$ be measurable. Then, for every $f\in L^r(\Omega)$,
	\begin{equation}\label{eq:weighted-fiber-localized}
		\|w_\theta\,\widetilde M_\theta f\|_{L^r(E)}^r
		\le
		\int_{\theta^{-1}(E)}
		|f(x)|^r\,w_\theta(\theta(x))^{r-1}\,\mathrm{d}x.
	\end{equation}
\end{proposition}

\begin{proof}
	By \eqref{eq:wtheta-times-Mtilde} and \eqref{eq:holder-fiber}, for almost every $t$ one has
	\[
	|w_\theta(t)\,\widetilde M_\theta f(t)|
	\le
	w_\theta(t)^{1/r'}\,
	\left(
	\int_{\Sigma_t}\frac{|f(x)|^r}{|\nabla\theta(x)|}\,\mathrm{d}\mathcal H^{n-1}(x)
	\right)^{1/r}.
	\]
	Raising to the power $r$ and integrating over $E$, we obtain
	\[
	\|w_\theta\,\widetilde M_\theta f\|_{L^r(E)}^r
	\le
	\int_E
	w_\theta(t)^{r-1}
	\int_{\Sigma_t}\frac{|f(x)|^r}{|\nabla\theta(x)|}\,\mathrm{d}\mathcal H^{n-1}(x)\,
	\mathrm{d}t.
	\]
	Applying the coarea formula \eqref{eq:coarea} to the integrand
	\[
	x\longmapsto |f(x)|^r\,w_\theta(\theta(x))^{r-1}\,\mathbf{1}_E(\theta(x)),
	\]
	gives exactly \eqref{eq:weighted-fiber-localized}.
\end{proof}

The preceding proposition identifies the analytic cost of the critical regime: instead of an unweighted bound on \(\Omega\), the pullback weight
\[
w_\theta(\theta(x))^{r-1}
\]
appears naturally.

Under the blow-up profile \eqref{eq:wtheta-critical-blowup}, this weight can be estimated separately in a critical zone, concentrated near \(V_\theta\), and in a noncritical zone, where local uniform control is recovered.

\begin{corollary}[Localization near $V_\theta$]\label{cor:weighted-fiber-critical-localization}
	Assume that $\Omega\subset\mathbb{R}^n$ is bounded, that $\theta:\Omega\to\mathbb{R}$ is Lipschitz,
	that $V_\theta$ is finite, that $w_\theta\in L^\infty_{\mathrm{loc}}(\mathbb{R}\setminus V_\theta)$,
	and that \eqref{eq:wtheta-critical-blowup} holds for some $\beta\in[0,1)$.
	Fix $\delta\in(0,\delta_0)$ and define
	\[
	U_\delta:=\{t\in\mathbb{R}:\operatorname{dist}(t,V_\theta)<\delta\}.
	\]
	Then, for every
	\[
	1<r<1+\frac{1}{\beta}
	\]
	and every $f\in L^r(\Omega)$, one has
	\begin{equation}\label{eq:weighted-fiber-critical-local}
		\|w_\theta\,\widetilde M_\theta f\|_{L^r(U_\delta)}^r
		\le
		C_{\theta,\beta}^{\,r-1}
		\int_{\theta^{-1}(U_\delta)}
		|f(x)|^r\,\operatorname{dist}(\theta(x),V_\theta)^{-\beta(r-1)}\,\mathrm{d}x,
	\end{equation}
	and moreover there exists a constant $C_{\theta,\delta}<\infty$ such that
	\begin{equation}\label{eq:weighted-fiber-critical-off}
		\|w_\theta\,\widetilde M_\theta f\|_{L^r(\mathbb{R}\setminus U_\delta)}
		\le
		C_{\theta,\delta}\,\|f\|_{L^r(\Omega)}.
	\end{equation}
\end{corollary}

\begin{proof}
	The inequality \eqref{eq:weighted-fiber-critical-local} follows immediately from Proposition~\ref{prop:weighted-fiber-localized} and from the pointwise bound \eqref{eq:wtheta-critical-blowup}.

	For \eqref{eq:weighted-fiber-critical-off}, by the boundedness of $\Omega$ and the Lipschitz character of $\theta$, the set $\theta(\Omega)$ is contained in a compact interval. The set
	\[
	K_\delta
	:=
	\theta(\Omega)\cap\{t\in\mathbb{R}:\operatorname{dist}(t,V_\theta)\ge\delta\}
	\]
	is compact and contained in $\mathbb{R}\setminus V_\theta$. Since $w_\theta\in L^\infty_{\mathrm{loc}}(\mathbb{R}\setminus V_\theta)$, there exists $C_{\theta,\delta}<\infty$ such that
	\[
	w_\theta(t)\le C_{\theta,\delta}
	\qquad\text{for almost every }t\in K_\delta.
	\]
	Applying Proposition~\ref{prop:weighted-fiber-localized} again,
	\[
	\|w_\theta\,\widetilde M_\theta f\|_{L^r(\mathbb{R}\setminus U_\delta)}^r
	\le
	C_{\theta,\delta}^{\,r-1}\,\|f\|_{L^r(\Omega)}^r,
	\]
	which gives \eqref{eq:weighted-fiber-critical-off}.
\end{proof}

\begin{remark}\label{rem:critical-regime-scope}
	The output obtained in this section is the critical counterpart of the uniform regime analyzed in Section~\ref{sec:fiber-analytic-objective}. There, the geometric stability of the fibers and the quantitative control of \(w_\theta\) allow the recomposition to close without critical cost. Here, instead, the loss of that stability near \(V_\theta\) forces one to work with a formulation localized in the level space and with a pullback weight on the input space.

	In particular, the critical regime does not represent a failure of the transfer mechanism, but rather a second structural mode of closure. Its distinctive feature is that the recomposition no longer produces, in general, a global unweighted estimate, but instead a localized and weighted output whose scope is governed by the singular profile of \(w_\theta\) near the critical set and by the exponent window fixed above.
\end{remark}

%% file: 40_main_thm_abs-sio.tex
\section{Sparse transfer and geometric recomposition}\label{sec:main-results}

This section proves the second main result of the manuscript. The goal is to show that the one-dimensional sparse output available for the reduced form transfers to the original geometric problem and recomposes into functional consequences on \(\Omega_y\times\Omega_x\). The proof has three structural steps: exact reduction to the level variables, sparse transfer from the one-dimensional model, and geometric closure of the pushforward norms.

\subsection{From the sparse input to abstract recomposition}\label{subsec:main-assembly}

The first step is purely functional. It assumes that a transferred sparse bound for the truncated geometric form is already available and converts it into an \(L^r(\mathbb{R})\times L^{r'}(\mathbb{R})\) bound on the pushforward densities. For this reason, we record it as a preparatory lemma.

\begin{lemma}[Abstract recomposition of a sparse output]\label{lem:abstract-recomposition}
	Let \(1<r<\infty\), and let \(r'\) be the conjugate exponent. Fix \(\varepsilon>0\) and a pair of bounded, compactly supported functions
	\[
	f:\Omega_y\to\mathbb{C},
	\qquad
	g:\Omega_x\to\mathbb{C}.
	\]
	Assume that there exists a sparse family \(\mathcal S=\mathcal S_{\varepsilon,f,g}\) of intervals such that
	\begin{equation}\label{eq:main-recomposition-sparse-input}
		\bigl|\langle T_\varepsilon f,g\rangle\bigr|
		\le
		C_{\mathrm{sp}}\,\Lambda_{\mathcal S}(A_\psi f,A_\phi g).
	\end{equation}
	Assume moreover that
	\[
	A_\psi f\in L^r(\mathbb{R}),
	\qquad
	A_\phi g\in L^{r'}(\mathbb{R}).
	\]
	Then
	\begin{equation}\label{eq:main-recomposition-lr-bound}
		\bigl|\langle T_\varepsilon f,g\rangle\bigr|
		\le
		C_r\,\|A_\psi f\|_{L^r(\mathbb{R})}\,
		\|A_\phi g\|_{L^{r'}(\mathbb{R})},
	\end{equation}
	where \(C_r>0\) depends only on \(r\), on \(C_{\mathrm{sp}}\), and on the universal constant in the boundedness of sparse forms on \(L^r(\mathbb{R})\times L^{r'}(\mathbb{R})\).
\end{lemma}

\begin{proof}
	By the boundedness of sparse forms on \(L^r(\mathbb{R})\times L^{r'}(\mathbb{R})\),
	\[
	\Lambda_{\mathcal S}(A_\psi f,A_\phi g)
	\le
	C_r'\,\|A_\psi f\|_{L^r(\mathbb{R})}\,
	\|A_\phi g\|_{L^{r'}(\mathbb{R})}.
	\]
	Substituting this inequality into \eqref{eq:main-recomposition-sparse-input}, we obtain \eqref{eq:main-recomposition-lr-bound}.
\end{proof}

\begin{lemma}[Globalization by finite covering of levels]
	\label{lem:main-uniform-globalization}
	Let \(1<r<\infty\), and let \(r'\) be the conjugate exponent. Let \(\theta:\Omega\to\mathbb{R}\), and assume that there exist open intervals
	\[
	I_1,\dots,I_N\subset\mathbb{R}
	\]
	such that
	\[
	\theta(\Omega)\subset \bigcup_{j=1}^N I_j.
	\]
	Assume moreover that:
	\begin{enumerate}
		\item for each \(j=1,\dots,N\), the local submersion and fiber-control hypotheses hold on \(I_j\);
		\item there exists a constant \(C_{\theta,\mathrm{sub}}>0\) such that
		\[
		w_\theta(t)\le C_{\theta,\mathrm{sub}}
		\qquad\text{for almost every }t\in \theta(\Omega).
		\]
	\end{enumerate}
	Then, for every \(h\in L^r(\Omega)\),
	\begin{equation}\label{eq:main-uniform-globalization}
		\|w_\theta\,\widetilde M_\theta h\|_{L^r(\mathbb{R})}
		\le
		N^{1/r}\,C_{\theta,\mathrm{sub}}^{1/r'}\,\|h\|_{L^r(\Omega)}.
	\end{equation}
\end{lemma}

\begin{proof}
	Since \(w_\theta\,\widetilde M_\theta h=0\) outside \(\theta(\Omega)\), one has
	\[
	\|w_\theta\,\widetilde M_\theta h\|_{L^r(\mathbb{R})}^r
	=
	\int_{\theta(\Omega)}
	|w_\theta(t)\,\widetilde M_\theta h(t)|^r\,\mathrm{d}t.
	\]
	Using the covering \(\theta(\Omega)\subset \bigcup_{j=1}^N I_j\),
	\[
	\|w_\theta\,\widetilde M_\theta h\|_{L^r(\mathbb{R})}^r
	\le
	\sum_{j=1}^N
	\|w_\theta\,\widetilde M_\theta h\|_{L^r(I_j)}^r.
	\]
	On each interval \(I_j\), the uniform fiber control gives
	\[
	\|w_\theta\,\widetilde M_\theta h\|_{L^r(I_j)}^r
	\le
	C_{\theta,\mathrm{sub}}^{r-1}\,\|h\|_{L^r(\Omega)}^r.
	\]
	Therefore,
	\[
	\|w_\theta\,\widetilde M_\theta h\|_{L^r(\mathbb{R})}^r
	\le
	N\,C_{\theta,\mathrm{sub}}^{r-1}\,\|h\|_{L^r(\Omega)}^r.
	\]
	Taking \(r\)-th roots yields \eqref{eq:main-uniform-globalization}.
\end{proof}

\subsection{The transfer and recomposition theorem}\label{subsec:main-theorem}

The following result is the second main theorem. Unlike the preceding abstract lemma, it states simultaneously the reduction hypotheses, the sparse input, and the two geometric modes of closure.

\begin{theorem}[Sparse transfer and geometric recomposition]\label{thm:main-recomposition}
	Let \(1<r<\infty\), and let \(r'\) be the conjugate exponent. Fix \(\varepsilon>0\), and let
	\[
	f:\Omega_y\to\mathbb{C},
	\qquad
	g:\Omega_x\to\mathbb{C}
	\]
	be bounded, compactly supported functions.

	Assume that the synchronized truncated form satisfies the reduction identity in the Lebesgue layer
	\begin{equation}\label{eq:main-thm-reduction-assumption}
		\langle T_\varepsilon f,g\rangle
		=
		\int_{\mathbb{R}}\int_{\mathbb{R}}
		\mathbf{1}_{\{|s-t|>\varepsilon\}}\,k(s,t)\,
		A_\psi f(t)\,A_\phi g(s)\,
		\mathrm{d}t\,\mathrm{d}s.
	\end{equation}
	Assume moreover that the reduced one-dimensional form admits the sparse output
	\begin{equation}\label{eq:main-thm-sparse-assumption}
		\bigl|\langle T_\varepsilon f,g\rangle\bigr|
		\le
		C_{\mathrm{sp}}\,
		\Lambda_{\mathcal S_{\varepsilon,f,g}}(A_\psi f,A_\phi g)
	\end{equation}
	for some sparse family \(\mathcal S_{\varepsilon,f,g}\) of intervals.

	Then the following conclusions hold.

	\begin{enumerate}
		\item \emph{Abstract recomposition.}
		If
		\[
		A_\psi f\in L^r(\mathbb{R}),
		\qquad
		A_\phi g\in L^{r'}(\mathbb{R}),
		\]
		then
		\begin{equation}\label{eq:main-recomposition-theorem-bound}
			\bigl|\langle T_\varepsilon f,g\rangle\bigr|
			\le
			C_r\,\|A_\psi f\|_{L^r(\mathbb{R})}\,
			\|A_\phi g\|_{L^{r'}(\mathbb{R})}.
		\end{equation}

		\item \emph{Uniform regime.}
		Assume that there exist finite families of open intervals
		\[
		\{I_{\phi,\ell}\}_{\ell=1}^{N_\phi},
		\qquad
		\{I_{\psi,m}\}_{m=1}^{N_\psi},
		\]
		and finite geometric constants
		\[
		C_{\phi,\mathrm{sub}},\ C_{\psi,\mathrm{sub}}>0
		\]
		such that
		\[
		\phi(\Omega_x)\subset \bigcup_{\ell=1}^{N_\phi} I_{\phi,\ell},
		\qquad
		\psi(\Omega_y)\subset \bigcup_{m=1}^{N_\psi} I_{\psi,m},
		\]
		the local submersion and fiber-control hypotheses hold on every interval in both families, and
		\[
		w_\phi(t)\le C_{\phi,\mathrm{sub}}
		\quad\text{for almost every }t\in \phi(\Omega_x),
		\qquad
		w_\psi(t)\le C_{\psi,\mathrm{sub}}
		\quad\text{for almost every }t\in \psi(\Omega_y).
		\]
		Then
		\begin{equation}\label{eq:main-uniform-final-bound}
			|\langle T_\varepsilon f,g\rangle|
			\le
			C_r\,
			N_\psi^{1/r}\,N_\phi^{1/r'}\,
			C_{\psi,\mathrm{sub}}^{1/r'}\,C_{\phi,\mathrm{sub}}^{1/r}\,
			\|f\|_{L^r(\Omega_y)}\,\|g\|_{L^{r'}(\Omega_x)}.
		\end{equation}

		\item \emph{Critical regime.}
		Assume that the set of critical values \(V_\psi\) is finite, that \(w_\psi\in L^\infty_{\mathrm{loc}}(\mathbb{R}\setminus V_\psi)\), and that there exist \(\delta_{\psi,0}>0\), \(\beta_\psi\in[0,1)\), and \(C_{\psi,\beta}>0\) such that
		\begin{equation}\label{eq:main-critical-blowup-psi}
			w_\psi(t)
			\le
			C_{\psi,\beta}\,\operatorname{dist}(t,V_\psi)^{-\beta_\psi}
			\qquad
			\text{for almost every }t\text{ with }\operatorname{dist}(t,V_\psi)<\delta_{\psi,0}.
		\end{equation}
		Assume analogously that \(V_\phi\) is finite, that \(w_\phi\in L^\infty_{\mathrm{loc}}(\mathbb{R}\setminus V_\phi)\), and that there exist \(\delta_{\phi,0}>0\), \(\beta_\phi\in[0,1)\), and \(C_{\phi,\beta}>0\) such that
		\begin{equation}\label{eq:main-critical-blowup-phi}
			w_\phi(t)
			\le
			C_{\phi,\beta}\,\operatorname{dist}(t,V_\phi)^{-\beta_\phi}
			\qquad
			\text{for almost every }t\text{ with }\operatorname{dist}(t,V_\phi)<\delta_{\phi,0}.
		\end{equation}
		Fix
		\[
		0<\delta_\psi<\delta_{\psi,0},
		\qquad
		0<\delta_\phi<\delta_{\phi,0},
		\]
		and define
		\[
		U_{\psi,\delta_\psi}
		:=
		\{t\in\mathbb{R}:\operatorname{dist}(t,V_\psi)<\delta_\psi\},
		\qquad
		U_{\phi,\delta_\phi}
		:=
		\{t\in\mathbb{R}:\operatorname{dist}(t,V_\phi)<\delta_\phi\}.
		\]
		Assume finally that
		\[
		\beta_\psi(r-1)<1,
		\qquad
		\beta_\phi(r'-1)<1.
		\]
		Then
		\begin{equation}\label{eq:main-critical-localized-bound}
			|\langle T_\varepsilon f,g\rangle|
			\le
			C_r\,
			\bigl(\mathcal N_{\psi,\delta_\psi}(f)+C_{\psi,\delta_\psi}\|f\|_{L^r(\Omega_y)}\bigr)
			\bigl(\mathcal N_{\phi,\delta_\phi}(g)+C_{\phi,\delta_\phi}\|g\|_{L^{r'}(\Omega_x)}\bigr),
		\end{equation}
		where
		\begin{align}
			\mathcal N_{\psi,\delta_\psi}(f)
			&:=
			C_{\psi,\beta}^{1/r'}\,
			\left(
			\int_{\psi^{-1}(U_{\psi,\delta_\psi})}
			|f(y)|^r\,
			\operatorname{dist}(\psi(y),V_\psi)^{-\beta_\psi(r-1)}\,
			\mathrm{d}y
			\right)^{1/r},\label{eq:main-critical-localized-Npsi}\\
			\mathcal N_{\phi,\delta_\phi}(g)
			&:=
			C_{\phi,\beta}^{1/r}\,
			\left(
			\int_{\phi^{-1}(U_{\phi,\delta_\phi})}
			|g(x)|^{r'}\,
			\operatorname{dist}(\phi(x),V_\phi)^{-\beta_\phi(r'-1)}\,
			\mathrm{d}x
			\right)^{1/r'}.\label{eq:main-critical-localized-Nphi}
		\end{align}
	\end{enumerate}
\end{theorem}

\begin{proof}
	We begin with the abstract part. By the reduction identity \eqref{eq:main-thm-reduction-assumption}, the truncated geometric form is written exactly as the reduced one-dimensional form applied to the densities \(A_\psi f\) and \(A_\phi g\). The sparse hypothesis \eqref{eq:main-thm-sparse-assumption} then gives
	\[
	\bigl|\langle T_\varepsilon f,g\rangle\bigr|
	\le
	C_{\mathrm{sp}}\,
	\Lambda_{\mathcal S_{\varepsilon,f,g}}(A_\psi f,A_\phi g).
	\]
	If \(A_\psi f\in L^r(\mathbb{R})\) and \(A_\phi g\in L^{r'}(\mathbb{R})\), the \(L^r\times L^{r'}\) boundedness of sparse forms implies
	\[
	\Lambda_{\mathcal S_{\varepsilon,f,g}}(A_\psi f,A_\phi g)
	\le
	C_r'\,\|A_\psi f\|_{L^r(\mathbb{R})}\,
	\|A_\phi g\|_{L^{r'}(\mathbb{R})}.
	\]
	Absorbing \(C_{\mathrm{sp}}C_r'\) into \(C_r\), we obtain \eqref{eq:main-recomposition-theorem-bound}.

	We now prove the uniform conclusion. In this regime one has the geometric identification
	\[
	A_\psi f=w_\psi\,\widetilde M_\psi f,
	\qquad
	A_\phi g=w_\phi\,\widetilde M_\phi g.
	\]
	Applying Lemma~\ref{lem:main-uniform-globalization} to \(\theta=\psi\), with exponent \(r\), gives
	\[
	\|A_\psi f\|_{L^r(\mathbb{R})}
	=
	\|w_\psi\,\widetilde M_\psi f\|_{L^r(\mathbb{R})}
	\le
	N_\psi^{1/r}\,C_{\psi,\mathrm{sub}}^{1/r'}\,
	\|f\|_{L^r(\Omega_y)}.
	\]
	Applying the same lemma to \(\theta=\phi\), with exponent \(r'\), and recalling that the conjugate exponent of \(r'\) is \(r\), gives
	\[
	\|A_\phi g\|_{L^{r'}(\mathbb{R})}
	=
	\|w_\phi\,\widetilde M_\phi g\|_{L^{r'}(\mathbb{R})}
	\le
	N_\phi^{1/r'}\,C_{\phi,\mathrm{sub}}^{1/r}\,
	\|g\|_{L^{r'}(\Omega_x)}.
	\]
	These two bounds show, in particular, that \(A_\psi f\in L^r(\mathbb{R})\) and \(A_\phi g\in L^{r'}(\mathbb{R})\). Substituting them into the abstract estimate \eqref{eq:main-recomposition-theorem-bound}, we obtain
	\[
	|\langle T_\varepsilon f,g\rangle|
	\le
	C_r\,
	N_\psi^{1/r}\,N_\phi^{1/r'}\,
	C_{\psi,\mathrm{sub}}^{1/r'}\,C_{\phi,\mathrm{sub}}^{1/r}\,
	\|f\|_{L^r(\Omega_y)}\,\|g\|_{L^{r'}(\Omega_x)},
	\]
	which is \eqref{eq:main-uniform-final-bound}.

	Finally, we prove the critical conclusion. Using again
	\[
	A_\psi f=w_\psi\,\widetilde M_\psi f,
	\qquad
	A_\phi g=w_\phi\,\widetilde M_\phi g,
	\]
	we split the norm of \(A_\psi f\) into the critical neighborhood and its complement:
	\[
	\|A_\psi f\|_{L^r(\mathbb{R})}
	\le
	\|w_\psi\,\widetilde M_\psi f\|_{L^r(U_{\psi,\delta_\psi})}
	+
	\|w_\psi\,\widetilde M_\psi f\|_{L^r(\mathbb{R}\setminus U_{\psi,\delta_\psi})}.
	\]
	On \(U_{\psi,\delta_\psi}\), the profile
	\[
	w_\psi(t)\le C_{\psi,\beta}\,\operatorname{dist}(t,V_\psi)^{-\beta_\psi}
	\]
	and H\"older's inequality on the fibers give
	\[
	\|w_\psi\,\widetilde M_\psi f\|_{L^r(U_{\psi,\delta_\psi})}
	\le
	\mathcal N_{\psi,\delta_\psi}(f).
	\]
	Outside \(U_{\psi,\delta_\psi}\), the hypothesis \(w_\psi\in L^\infty_{\mathrm{loc}}(\mathbb{R}\setminus V_\psi)\) and the positive separation from \(V_\psi\) give a finite constant \(C_{\psi,\delta_\psi}\) such that
	\[
	\|w_\psi\,\widetilde M_\psi f\|_{L^r(\mathbb{R}\setminus U_{\psi,\delta_\psi})}
	\le
	C_{\psi,\delta_\psi}\,\|f\|_{L^r(\Omega_y)}.
	\]
	Therefore,
	\[
	\|A_\psi f\|_{L^r(\mathbb{R})}
	\le
	\mathcal N_{\psi,\delta_\psi}(f)+C_{\psi,\delta_\psi}\|f\|_{L^r(\Omega_y)}.
	\]
	The same argument applied to \(\phi\), now with exponent \(r'\), gives
	\[
	\|A_\phi g\|_{L^{r'}(\mathbb{R})}
	\le
	\mathcal N_{\phi,\delta_\phi}(g)+C_{\phi,\delta_\phi}\|g\|_{L^{r'}(\Omega_x)}.
	\]
	These two estimates imply the required membership of \(A_\psi f\) and \(A_\phi g\) in the corresponding Lebesgue spaces. Substituting them into \eqref{eq:main-recomposition-theorem-bound}, we conclude \eqref{eq:main-critical-localized-bound}.
\end{proof}

\begin{corollary}[Global output in the uniform regime]
	\label{cor:main-uniform}
	Under the uniform hypotheses of item \emph{(2)} of Theorem~\ref{thm:main-recomposition}, the global estimate \eqref{eq:main-uniform-final-bound} holds.
\end{corollary}

\begin{proof}
	This is exactly the uniform conclusion of Theorem~\ref{thm:main-recomposition}.
\end{proof}

\begin{corollary}[Localized recomposition in the critical regime]\label{cor:main-critical-localized}
	Under the critical hypotheses of item \emph{(3)} of Theorem~\ref{thm:main-recomposition}, the localized estimate \eqref{eq:main-critical-localized-bound} holds.
\end{corollary}

\begin{proof}
	This is exactly the critical conclusion of Theorem~\ref{thm:main-recomposition}.
\end{proof}

\begin{remark}[Constants, localization, and critical window]
	\label{rem:main-uniform-constant-bookkeeping}\label{rem:main-critical-localized-scope}
	In the global uniform inequality \eqref{eq:main-uniform-final-bound}, the constant \(C_r\) contains the one-dimensional analytic dependence, while the geometry of the recomposition is recorded by
	\[
	N_\psi^{1/r}\,N_\phi^{1/r'}\,C_{\psi,\mathrm{sub}}^{1/r'}\,C_{\phi,\mathrm{sub}}^{1/r}.
	\]
	This factor measures the structural cost of the local-to-global passage by finite covering. In the critical regime, the estimate \eqref{eq:main-critical-localized-bound} replaces that global closure with localized norms carrying pullback weights determined by the blow-up profiles of \(w_\psi\) and \(w_\phi\). The conditions
	\[
	\beta_\psi(r-1)<1,
	\qquad
	\beta_\phi(r'-1)<1
	\]
	are equivalent to the window
	\[
	1+\beta_\phi<r<1+\frac{1}{\beta_\psi},
	\]
	with endpoints excluded.
\end{remark}

\begin{remark}[Scope of the recomposition]\label{rem:main-pv-maximal-status}
	The observations concerning maximal truncation and principal values are treated separately in the complementary capsule of Section~\ref{sec:pv-maximal-capsule}. The functional closure of the two regimes recomposed in this section is formulated for the truncated family \(T_\varepsilon\), which is the main object of Theorem~\ref{thm:main-recomposition}.
\end{remark}

%% file: 50_examples_abs-sio.tex
\section{Examples and geometric regimes}\label{sec:examples-regimes}

In this section we collect fully explicit models illustrating the two geometric outputs of Theorem~\ref{thm:main-recomposition}: the uniform output of Corollary~\ref{cor:main-uniform} and the localized output of Corollary~\ref{cor:main-critical-localized}. In particular, the examples separate three layers that should not be confused: the presence of critical values, the geometric uniformity of the fibers, and the possibility of a global functional consequence.

\subsection{Laboratory convention}\label{subsec:examples-setup}

Unless explicitly stated otherwise, we work in the laboratory setting
\[
\Omega=B(0,R)\subset\mathbb{R}^n,
\]
and take as one-dimensional kernel the Hilbert kernel
\begin{equation*}\label{eq:examples-hilbert-kernel}
k(s,t)=\frac{1}{\pi}\,\frac{1}{s-t}.
\end{equation*}
This kernel belongs to the one-dimensional analytic package fixed in Subsection~\ref{subsec:kernel-1d}, so that the following examples explicitly realize the abstract assembly of Section~\ref{sec:main-results}.

\subsection{Uniform regime: linear projection on the ball}\label{subsec:examples-uniform-linear}

Let
\[
\theta(x)=x_1.
\]
Then $\nabla\theta\equiv e_1$, so there are no critical points and the fibers
\[
\Sigma_t=\{x\in\mathbb{R}^n:\ x_1=t\}
\]
are parallel hyperplanes.
The intersection with the ball is an $(n-1)$-dimensional disk of radius $\sqrt{R^2-t^2}$ for $|t|<R$; hence the pushforward density is explicit:
\begin{equation*}\label{eq:w-linear-ball}
w_\theta(t)
=
\mathcal{H}^{n-1}(\Sigma_t\cap B(0,R))
=
\omega_{n-1}(R^2-t^2)^{\frac{n-1}{2}}\mathbf{1}_{\{|t|\le R\}}.
\end{equation*}

In particular,
\[
w_\theta\in L^\infty(\mathbb{R}),
\]
and therefore the hypothesis $w_\theta\in L^\infty(\mathbb{R})$ of Proposition~\ref{prop:pushforward-linfty-to-lp} is immediately verified.
Moreover, this example satisfies the prototypical geometry of the uniform regime:
the submersion is global, the fibers admit a natural quantitative trivialization, and,
away from the extreme levels $t=\pm R$, the contact with the boundary is transverse.

Indeed, if $n(x)=x/R$ denotes the outward normal to $\partial B(0,R)$, then for
$x\in \Sigma_t\cap\partial\Omega$ one has
\begin{equation*}\label{eq:angle-linear-ball}
|\nabla_{\partial\Omega}\theta(x)|
=
\sqrt{1-\Bigl(\frac{t}{R}\Bigr)^2},
\end{equation*}
so that, in interior tubes $|t|\le (1-\delta)R$, the transversality is quantitatively separated from zero.

Consequently, this model falls directly within the scope of Corollary~\ref{cor:wtheta-submersion-uniform} and activates the uniform output of Corollary~\ref{cor:main-uniform}. It is the canonical model of global recomposition without geometric loss.

\subsection{Critical values and profiles of \texorpdfstring{$w_\theta$}{w theta}}\label{subsec:examples-critical-profiles}

The next two examples separate two different possibilities within the critical regime. The first shows that critical values may occur without blow-up of the pushforward density. The second shows that, when a singular profile does occur, it need not be of power type.

\subsubsection{Critical value without blow-up: the quadratic radial model}\label{subsubsec:examples-critical-radial}

Let now
\[
\theta(x)=|x|^2,
\qquad x\in B(0,R)\subset\mathbb{R}^n.
\]
Then
\[
\nabla\theta(x)=2x,
\]
so a critical point appears at $x=0$ and, consequently, a critical value appears at $t=0$.
The levels are the spheres
\[
\Sigma_t=\{x:\ |x|^2=t\},
\qquad 0<t\le R^2,
\]
and the pushforward density is computed explicitly by coarea:
\begin{equation*}\label{eq:w-radial-ball}
w_\theta(t)
=
\int_{\Sigma_t\cap B(0,R)}\frac{1}{|\nabla\theta|}\,\mathrm{d}\mathcal{H}^{n-1}
=
\frac{\sigma_{n-1}}{2}\,t^{\frac{n-2}{2}}\,\mathbf{1}_{\{0<t\le R^2\}},
\end{equation*}
where $\sigma_{n-1}=\mathcal{H}^{n-1}(S^{n-1}_1)$.

This example is important precisely because it separates two phenomena that should not be confused:
the presence of critical values and the blow-up of the pushforward density.
Indeed, for $n\ge 2$ there is no blow-up of $w_\theta$ at $t=0$:
if $n=2$, $w_\theta$ is constant near $0$;
if $n>2$, one even has
\[
w_\theta(t)\to 0
\qquad \text{as } t\downarrow 0.
\]
In particular, for $n\ge 2$ one also obtains
\[
w_\theta\in L^\infty(\mathbb{R}),
\]
so the abstract criterion of Proposition~\ref{prop:pushforward-linfty-to-lp}
is not excluded by the mere presence of the critical value.

This example separates the presence of a critical value from the blow-up of the pushforward density. For \(n\ge 2\), \(w_\theta\) remains bounded near \(t=0\), so the mere existence of a critical value does not by itself exclude a global functional consequence. What is lost is the purely uniform reading of the regime.

\subsubsection{Critical value with logarithmic blow-up: the flat saddle point}\label{subsubsec:examples-critical-saddle}

Consider now the two-dimensional model
\[
\Omega=B(0,1)\subset\mathbb{R}^2
\]
and the quadratic phase
\[
\theta(x,y)=x^2-y^2.
\]
Then
\[
\nabla\theta(x,y)=(2x,-2y),
\]
so the origin is the unique critical point and therefore
\[
V_\theta=\{0\}.
\]
For $t\neq 0$, the fibers
\[
\Sigma_t=\{(x,y)\in\mathbb{R}^2:\ x^2-y^2=t\}
\]
are hyperbolas; the critical level $\Sigma_0$ degenerates into the pair of lines $y=\pm x$.

The pushforward density can be computed explicitly by coarea.
By symmetry, it is enough to consider the case $0<t<1$ and multiply by four the contribution from the first quadrant.
There the fiber is parametrized as
\[
x=\sqrt{t+y^2},
\qquad
0\le y\le \sqrt{\frac{1-t}{2}},
\]
since the condition $x^2+y^2\le 1$ is equivalent to $t+2y^2\le 1$.
Moreover,
\[
|\nabla\theta(x,y)|=2\sqrt{x^2+y^2}=2\sqrt{t+2y^2},
\]
and the length element along the curve is
\[
\mathrm{d}\mathcal H^1
=
\sqrt{1+\Bigl(\frac{\mathrm{d}x}{\mathrm{d}y}\Bigr)^2}\,\mathrm{d}y
=
\frac{\sqrt{t+2y^2}}{\sqrt{t+y^2}}\,\mathrm{d}y.
\]
Therefore,
\begin{align}
w_\theta(t)
&=
\int_{\Sigma_t\cap B(0,1)}\frac{1}{|\nabla\theta|}\,\mathrm{d}\mathcal H^1 \notag\\
&=
4\int_0^{\sqrt{(1-t)/2}}
\frac{1}{2\sqrt{t+2y^2}}\,
\frac{\sqrt{t+2y^2}}{\sqrt{t+y^2}}
\,\mathrm{d}y \notag\\
&=
2\int_0^{\sqrt{(1-t)/2}}\frac{1}{\sqrt{t+y^2}}\,\mathrm{d}y \notag\\
&=
2\,\operatorname{arcsinh}\!\left(\sqrt{\frac{1-t}{2t}}\right).\notag
\label{eq:w-saddle-ball}
\end{align}
By symmetry, the same expression holds for $t<0$ with $t$ replaced by $|t|$, and $w_\theta(t)=0$ for $|t|\ge 1$.

In particular, as $t\to 0$,
\begin{equation*}\label{eq:w-saddle-log}
w_\theta(t)
\sim
\log\frac{1}{|t|}.
\end{equation*}
Thus this example exhibits a genuine critical value and an effective loss of geometric uniformity, but the blow-up of the pushforward density is only logarithmic.
Consequently,
\[
w_\theta\in L^a_{\mathrm{loc}}(\mathbb{R})
\qquad\text{for every }1\le a<\infty,
\]
although
\[
w_\theta\notin L^\infty_{\mathrm{loc}}(\mathbb{R}).
\]

This example complements the previous radial model: there may be critical values without blow-up there, while here blow-up appears without a strong loss of integrability. In both cases, the profile of \(w_\theta\), the geometric uniformity of the fibers, and the functional output belong to distinct layers of the analysis.

\subsection{Loss of uniformity without critical points: oscillation and boundary}\label{subsec:examples-loss-uniformity}

Not every loss of geometric uniformity comes from critical values. A different mechanism appears when the fibers remain smooth and submersive, but their geometry ceases to be quantitatively uniform.

An elementary example is
\begin{equation*}\label{eq:theta-oscillatory}
\theta(x)=x_1+a\sin(Nx_2),
\qquad x\in\mathbb{R}^n.
\end{equation*}
In this case
\[
\nabla\theta(x)=(1,aN\cos(Nx_2),0,\dots,0),
\]
and therefore
\[
|\nabla\theta(x)|\ge 1
\qquad \text{for every }x,
\]
so there are no critical points and the submersion persists globally.

However, as $N$ grows, the fibers
\[
\Sigma_t=\{x_1=t-a\sin(Nx_2)\}
\]
exhibit fine-scale oscillation, large curvature, and increasing complexity in their intersection with the domain and with the boundary. In families with \(N\to\infty\), the quantitative constants associated with the trivialization of the tube and with the contact with the boundary may degenerate.

This example has a diagnostic role: it shows that the absence of critical points does not by itself imply quantitative uniformity of the fibers.

\subsubsection{Tangential contact with the boundary}\label{subsubsec:examples-boundary-contact}

An additional mechanism of geometric degeneration appears when the fibers remain regular in the interior but lose transversality when they intersect the boundary of the domain. This phenomenon is distinct from the appearance of interior critical values and must be analyzed separately.

Consider, in the plane, the domain
\[
\Omega=B(0,R)\subset\mathbb{R}^2,
\]
and the linear function
\begin{equation*}\label{eq:theta-boundary-flat}
\theta(x)=x_1.
\end{equation*}
There is no interior degeneration, since
\[
\nabla\theta\equiv (1,0)\neq 0.
\]
However, as one approaches the extreme levels $t=\pm R$, the fiber
\[
\Sigma_t=\{x\in\mathbb{R}^2:\ x_1=t\}
\]
becomes tangent to the boundary $\partial B(0,R)$.
Indeed, the length of the transverse section is
\begin{equation*}\label{eq:w-boundary-flat}
w_\theta(t)
=
2\sqrt{R^2-t^2}\,\mathbf{1}_{\{|t|\le R\}},
\end{equation*}
which collapses as $|t|\uparrow R$.

Equivalently, if $n(x)=x/R$ denotes the outward normal to $\partial\Omega$, then for
$x\in\Sigma_t\cap\partial\Omega$ one has
\[
|\nabla_{\partial\Omega}\theta(x)|
=
\sqrt{1-\Bigl(\frac{t}{R}\Bigr)^2},
\]
so the transversality constant tends to zero precisely as one approaches the extreme levels.

This example shows that even a completely regular global submersion in the interior may lose uniformity
through a purely boundary mechanism.
In particular, the geometric hypotheses of the uniform regime must control not only the interior non-degeneracy, but also the quantitative contact between the fibers and the boundary of the domain, as already occurs in the interior tubes of the linear example in Subsection~\ref{subsec:examples-uniform-linear}.

\subsection{Exact radial laboratory: explicit reduction and critical threshold}\label{subsec:examples-radial-laboratory}

In this subsection we collect the full radial development. In this family, the abstract machinery of reduction and recomposition becomes completely explicit.

This makes it possible to see, within a single family, how the pushforward density arises, how the critical exponent \(\beta\) appears, and how the functional window of the critical regime narrows.

Let $\Omega_x=\Omega_y=B_R(0)\subset \mathbb{R}^n$, and assume that the phases are radial,
\[
\phi(x)=\Phi(|x|),\qquad \psi(y)=\Psi(|y|),
\]
with $\Phi,\Psi:[0,R]\to \mathbb{R}$ continuous, strictly monotone on $(0,R)$, and such that $\Phi,\Psi\in C^1((0,R])$.
Assume moreover that the inputs are radial,
\[
g(x)=G(|x|),\qquad f(y)=F(|y|).
\]

In what follows, \(s\) and \(t\) range over the images of the phases, and the identities are understood at regular values.

\begin{proposition}[Explicit reduction in the radial regime]\label{prop:radial-full-reduction}
Under the preceding hypotheses, if
\[
r_\phi(s):=\Phi^{-1}(s),\qquad r_\psi(t):=\Psi^{-1}(t),
\]
then the basic pushforward densities are given by
\[
w_\phi(s)=\omega_{n-1}\frac{r_\phi(s)^{n-1}}{|\Phi'(r_\phi(s))|},
\qquad
w_\psi(t)=\omega_{n-1}\frac{r_\psi(t)^{n-1}}{|\Psi'(r_\psi(t))|},
\]
and the densities weighted by radial data satisfy
\[
w_{\phi,g}(s)=w_\phi(s)G(r_\phi(s)),
\qquad
w_{\psi,f}(t)=w_\psi(t)F(r_\psi(t)).
\]
Consequently, the truncated form reduces exactly to
\[
\langle T_\varepsilon f,g\rangle =
\iint_{\mathbb{R}^2}
\mathbf{1}_{\{|s-t|>\varepsilon\}}
k(s,t)
w_\psi(t)F(r_\psi(t))
w_\phi(s)G(r_\phi(s))
\mathrm{d}t\,\mathrm{d}s.
\]
\end{proposition}

\begin{proof}
For each regular value $s$ of $\phi$, the fiber $\Sigma_s^\phi=\{x\in \Omega_x:\phi(x)=s\}$ is the sphere $\{|x|=r_\phi(s)\}$, and on it $|\nabla \phi(x)|=|\Phi'(r_\phi(s))|$ is constant. By the coarea formula,
\[
w_\phi(s) =
\int_{\Sigma_s^\phi}\frac{1}{|\nabla \phi(x)|}\,\mathrm{d}\mathcal H^{n-1}(x) =
\frac{\mathcal H^{n-1}(\{|x|=r_\phi(s)\})}{|\Phi'(r_\phi(s))|} =
\omega_{n-1}\frac{r_\phi(s)^{n-1}}{|\Phi'(r_\phi(s))|}.
\]
Similarly,
\[
w_\psi(t)=\omega_{n-1}\frac{r_\psi(t)^{n-1}}{|\Psi'(r_\psi(t))|}.
\]
Since $g$ is radial, $g$ is constant on each sphere $\Sigma_s^\phi$, so
\[
w_{\phi,g}(s) =
\int_{\Sigma_s^\phi} g(x)\frac{1}{|\nabla \phi(x)|}\,\mathrm{d}\mathcal H^{n-1}(x) =
G(r_\phi(s))
\int_{\Sigma_s^\phi}\frac{1}{|\nabla \phi(x)|}\,\mathrm{d}\mathcal H^{n-1}(x) =
w_\phi(s)G(r_\phi(s)).
\]
Analogously, $w_{\psi,f}(t)=w_\psi(t)F(r_\psi(t))$. The final identity then follows by substituting these expressions into the reduced Lebesgue formulation of Corollary~\ref{cor:reduction-lebesgue}.
\end{proof}

In particular, if $\phi=\psi=\theta$, then $r_\phi=r_\psi=r_\theta$ and $w_\phi=w_\psi=w_\theta$, and the reduced form takes the symmetric form
\[
\langle T_\varepsilon f,g\rangle =
\iint_{\mathbb{R}^2}
\mathbf{1}_{\{|s-t|>\varepsilon\}}
k(s,t)
w_\theta(t)F(r_\theta(t))
w_\theta(s)G(r_\theta(s))
\mathrm{d}t\,\mathrm{d}s.
\]

Likewise, if $h(x)=H(|x|)$ is radial, the fiber operator collapses to
\[
A_\theta h(t) =
\int_{\Sigma_t^\theta} h(x)\frac{1}{|\nabla \theta(x)|}\,\mathrm{d}\mathcal H^{n-1}(x) =
w_\theta(t)H(r_\theta(t)).
\]

This shows that, in the full radial regime, the geometric recomposition has no implicit content left: it reduces to multiplication by the pushforward density and composition with the radial inverse of the phase.

\paragraph{Toy model: pure polynomial flattening.}

Take now
\[
\phi(x)=|x|^\gamma,
\qquad \gamma>1,
\]
in $B_R(0)$.
Then
\[
\Phi(r)=r^\gamma,
\qquad
r_\phi(s)=s^{1/\gamma},
\qquad
|\Phi'(r)|=\gamma r^{\gamma-1}.
\]
Evaluating the gradient on the fiber $\Sigma_s^\phi=\{|x|=s^{1/\gamma}\}$,
\[
|\Phi'(r_\phi(s))| =
\gamma (s^{1/\gamma})^{\gamma-1} =
\gamma s^{1-\frac{1}{\gamma}}.
\]
Substituting into the explicit formula for $w_\phi$,
\[
w_\phi(s) =
\omega_{n-1}\frac{(s^{1/\gamma})^{n-1}}{\gamma s^{1-\frac{1}{\gamma}}} =
\frac{\omega_{n-1}}{\gamma}
s^{\frac{n-1}{\gamma}-\left(1-\frac{1}{\gamma}\right)} =
\frac{\omega_{n-1}}{\gamma}
s^{\frac{n-\gamma}{\gamma}}.
\]
Therefore,
\begin{equation*}\label{eq:radial-toy-weight}
w_\phi(s)=\frac{\omega_{n-1}}{\gamma}s^{\frac{n-\gamma}{\gamma}}.
\end{equation*}

This formula displays an exact geometric competition between the dimensional collapse of the sphere and the flattening of the phase at the origin.
In particular, three regimes appear.

\begin{itemize}
\item If $\gamma<n$, then
\[
\frac{n-\gamma}{\gamma}>0,
\]
and therefore
\[
w_\phi(s)\to 0
\qquad \text{as } s\to 0.
\]
In this regime, the radial phase introduces no critical obstruction from the geometric side: the collapse of the fibers dominates the vanishing of the gradient.

\item If $\gamma=n$, then
\[
w_\phi(s)\equiv \frac{\omega_{n-1}}{n}
\]
near the origin. This is the balanced case, where dimensional collapse and flattening cancel exactly.

\item If $\gamma>n$, then
\[
\frac{n-\gamma}{\gamma}<0,
\]
and $w_\phi$ exhibits a genuine blow-up at $s=0$. More precisely,
\[
w_\phi(s)\sim s^{-\beta_\phi},
\qquad
\beta_\phi=\frac{\gamma-n}{\gamma}=1-\frac{n}{\gamma}\in (0,1).
\]
This toy model realizes exactly the critical profile predicted by the abstract theory.
\end{itemize}

\begin{remark}[Critical threshold and narrowing of the functional window]\label{rem:radial-critical-window}
Assume now that both phases are equal and have the same radial degeneration,
\[
\phi(x)=|x|^\gamma,
\qquad
\psi(y)=|y|^\gamma,
\qquad
\gamma>n.
\]
Then
\[
\beta_\phi=\beta_\psi=\beta=1-\frac{n}{\gamma}.
\]
Substituting this exponent into the hypotheses of Corollary~\ref{cor:main-critical-localized}, one obtains the functional window
\[
1+\beta<r<1+\frac{1}{\beta},
\]
that is,
\[
2-\frac{n}{\gamma}
<
r
<
1+\frac{\gamma}{\gamma-n}
=
\frac{2\gamma-n}{\gamma-n}.
\]
For each $\gamma<\infty$, this window remains open and nonempty. However, as $\gamma\to\infty$, one has $\beta\to 1^{-}$ and therefore
\[
1+\beta\to 2,
\qquad
1+\frac{1}{\beta}\to 2.
\]
Consequently, the window of admissible exponents narrows asymptotically toward the Hilbert exponent \(r=2\). The functional restriction of the critical regime thus emerges from the competition between dimensional collapse of the fibers and extreme flattening of the phase.
\end{remark}

\begin{remark}[Scope of the examples]\label{rem:examples-scope}
	Taken together, these examples show that the presence of critical values, the geometric uniformity of the fibers, and the form of the final functional output belong to distinct layers of the analysis. The profile of the pushforward density \(w_\theta\) near \(V_\theta\) may remain bounded, display a weak logarithmic blow-up, or exhibit an explicit power law in degenerate radial models.
\end{remark}